\documentclass[a4paper]{article}
\usepackage{mathptmx}
\usepackage{amscd,amssymb,amsmath,amsthm}
\usepackage[T1]{fontenc} 

\makeatletter

\renewcommand{\@seccntformat}[1]
{{\csname the#1\endcsname}.\hspace{0.3em}}

\renewcommand{\section}{\@startsection
{section}
{1}
{0mm}
{-1.5\baselineskip}
{\baselineskip}
{\bfseries\normalsize}}

\renewcommand{\subsection}{\@startsection
{subsection}
{2}
{0mm}
{-\baselineskip}
{0.5\baselineskip}
{\normalsize\itshape}}

\renewcommand{\subsubsection}{\@startsection
{subsubsection}
{3}
{0mm}
{-.5\baselineskip}
{-2mm}
{\normalsize\itshape}}

\makeatother

\theoremstyle{plain}
\newtheorem*{theorem*}{Theorem}
\newtheorem{theorem}{Theorem}[section]
\newtheorem{lemma}{Lemma}[section]
\newtheorem{corollary}[lemma]{Corollary}
\newtheorem{prop}[lemma]{Proposition}
\newtheorem{claim}[lemma]{Claim}
\newtheorem*{corollary*}{Corollary}
\newtheorem*{claim*}{Claim}

\newtheorem*{BI}{Berger's inequality}

\theoremstyle{definition}
\newtheorem*{defin*}{Definition}
\newtheorem{defin}{Definition}[section]

\theoremstyle{remark}
\newtheorem{remark}{Remark}[section]

\DeclareMathAlphabet{\matheur}{U}{eur}{m}{n}
\DeclareMathAlphabet{\matheus}{U}{eus}{m}{n}
\DeclareMathAlphabet{\matheuf}{U}{euf}{m}{n}

\numberwithin{equation}{section}

\newcommand{\abs}[1]{\left\lvert#1\right\rvert}

\DeclareMathOperator{\Hess}{Hess}

\DeclareMathOperator{\dist}{dist}
\DeclareMathOperator{\sn}{sn}
\DeclareMathOperator{\inj}{inj}

\DeclareMathOperator{\rad}{rad}
\DeclareMathOperator{\conv}{conv}

\DeclareMathOperator{\grad}{grad}

\begin{document}

\author{Gerasim  Kokarev
\\ {\small\it School of Mathematics, The University of Leeds}
\\ {\small\it Leeds, LS2 9JT, United Kingdom}
\\ {\small\it Email: {\tt G.Kokarev@leeds.ac.uk}}
}

\title{The Berger inequality  for Riemannian manifolds with an upper sectional curvature bound}
\date{}
\maketitle

\begin{abstract}

\noindent
We obtain inequalities for all Laplace eigenvalues of Riemannian manifolds with an upper sectional curvature bound, whose rudiment version for the first Laplace eigenvalue was discovered by Berger in 1979. We show that our inequalities continue to hold for conformal metrics, and moreover, extend naturally to minimal submanifolds. In addition, we obtain explicit upper bounds for Laplace eigenvalues of minimal submanifolds in terms of geometric quantities of the ambient space.
\end{abstract}

\medskip
\noindent
{\small

 }

%
%\medskip
%\noindent
%{\small
%{\bf Mathematics Subject Classification (2000):} ...
%
%\noindent
%{\bf Keywords}:}
%

%\tableofcontents

\section{Statement and discussion of results}
\label{intro}

\subsection{Introduction}

Let $(M,g)$ be a closed Riemannian manifold of dimension $m$, and $\inj(g)$ its injectivity radius. A classical result by Berger~\cite{Be79} in 1979 says that for every $0<r<\inj(g)$ there exists a point $p\in M$ such that the first non-zero Dirichlet eigenvalue of a geodesic ball $B(p,r)$ in $M$ satisfies the inequality
$$
\lambda_0(B(p,r))\leqslant C_1(m)\frac{\mathit{Vol}_g(M)}{r^{m+2}},
$$
where $C_1(m)$ is a positive constant that depends on the dimension $m$ only. He uses this inequality to obtain the following upper bound for the first non-zero Laplace eigenvalue of $M$.
\begin{BI}
Let $(M,g)$ be a closed Riemannian manifold that admits an involutive isometry without fixed points. Then its first non-zero Laplace eigenvalue satisfies the inequality
\begin{equation}
\label{Be:1}
\lambda_1(g)\leqslant C_2(m)\frac{\mathit{Vol}_g(M)}{\inj(g)^{m+2}},
\end{equation}
where $C_2(m)$ is a constant that depends on the dimension $m$ only.
\end{BI}
In~\cite{Be79} Berger asks under what other geometric hypotheses on $M$ inequality~\eqref{Be:1} may hold. First answers are given by B\'erard and Besson~\cite{BB80}, who show that this inequality holds for homogeneous Riemannian manifolds and locally harmonic spaces. In a seminal paper~\cite{Cro80} Croke proves, among other results, a version of inequality~\eqref{Be:1} that uses the convexity radius $\conv(g)$ instead of $\inj(g)$, and holds for {\em arbitrary} closed Riemannian manifolds. More precisely, he shows that
\begin{equation}
\label{Cro:1}
\lambda_1(g)\leqslant C_3(m)\frac{\mathit{Vol}_g(M)^2}{\conv(g)^{2m+2}}.
\end{equation} 
The argument in~\cite{Cro80} uses a slightly different (to the one above) estimate for the first non-zero Dirichlet eigenvalue of geodesic balls, which actually yields inequalities for all Laplace eigenvalues
\begin{equation}
\label{Cro:2}
\lambda_k(g)\leqslant C_3(m)\frac{\mathit{Vol}_g(M)^2}{\conv(g)^{2m+2}}k^{2m},
\end{equation}
where $k\geqslant 1$ is an arbitrary integer. These inequalities for higher eigenvalues do not seem to appear in the literature, and we refer to Appendix~\ref{app:a} for related details.

%The above inequalities are universal in the sense that they do not require any curvature hypothesis on $M$. 

The purpose of this paper is to prove a version of the Berger inequality~\eqref{Be:1} for all Laplace eigenvalues of Riemannian manifolds with an upper sectional curvature bound, and their minimal submanifolds. For example, we show that inequality~\eqref{Be:1}, as well as its neat version for higher Laplace eigenvalues, holds for manifolds of non-positive sectional curvature. More importantly, we show that these eigenvalue inequalities are {\em conformal} in nature, that is the ratio $\mathit{Vol_g(M)}/\inj(g)^{m}$ controls Laplace eigenvalues of all metrics conformal to $g$. We also discover another interesting feature of these inequalities -- they are naturally {\em inherited my minimal submanifolds} in $M$. Below we discuss the results in detail.

\subsection{Conformal nature of the Berger inequality}
Let $(M,g)$ be a closed $m$-dimensional Riemannian manifold whose sectional curvatures are not greater than $\delta\geqslant 0$. By $\rad(g)$ we denote the quantity $\min\{\inj(g),\pi/(2\sqrt{\delta})\}$. When $\delta=0$, we always assume that $\pi/(2\sqrt{\delta})$ equals $+\infty$, and hence, $\rad(g)$ coincides with the injectivity radius $\inj(g)$. Further, we denote by 
$$
0=\lambda_0(g)<\lambda_1(g)\leqslant\lambda_2(g)\leqslant\ldots\leqslant\lambda_k(g)\leqslant\ldots
$$
the Laplace eigenvalues of a metric $g$ on $M$ repeated with respect to multiplicity. Our first result gives the following conformal bounds for all Laplace eigenvalues.
\begin{theorem}
\label{mt}
Let $(M,g)$ be a closed Riemannian manifold whose all sectional curvatures are not greater than $\delta\geqslant 0$. Then for any Riemannian metric $\tilde g$ conformal to $g$ its Laplace eigenvalues satisfy the inequalities
$$
\lambda_k(\tilde g)\mathit{Vol}_{\tilde g}(M)^{2/m}\leqslant C_4(m)\left(\frac{\mathit{Vol}_g(M)}{\rad(g)^m}\right)^{1+2/m}k^{2/m}
$$
for any $k\geqslant 1$, where $\rad(g)$ equals $\min\{\inj(g),\pi/(2\sqrt{\delta})\}$, and $C_3(m)$ is a positive constant that depends on the dimension $m$ of $M$ only. In particular, the Laplace eigenvalues of the metric $g$ satisfy the inequalities
\begin{equation}
\label{Be:3}
\lambda_k(g)\leqslant C_4(m)\frac{\mathit{Vol}_g(M)}{\rad(g)^{m+2}}k^{2/m}
\end{equation}
for any $k\geqslant 1$.
\end{theorem}

Note that under the hypotheses of Theorem~\ref{mt}, even the inequality for the first non-zero Laplace eigenvalue in~\eqref{Be:3} seems to be absent in the literature. When the sectional curvatures of $M$ are non-positive, inequalities~\eqref{Be:3} give a neat generalisation of the Berger inequality, and improve Croke's inequalities~\eqref{Cro:2}. To our knowledge, it is unknown whether in the absence of a curvature hypothesis the power $k^{2m}$ in inequalities~\eqref{Cro:2} can be replaced by the asymptotically sharp power $k^{2/m}$.

Recall that a celebrated result by Korevaar~\cite{Kor} says that for any closed $m$-dimensional Riemannian manifold $(M,\tilde g)$ its Laplace eigenvalues satisfy the inequalities
$$
\lambda_k(\tilde g)\mathit{Vol}_{\tilde g}(M)^{2/m}\leqslant C k^{2/m},
$$
where $C$ is the constant that depends on the conformal class of a metric $\tilde g$ in a rather implicit way. Thus, Theorem~\ref{mt} can be viewed as an explicit version of Korevaar's result that describes the dependance of the constant on the geometry of a background metric $g$ in a given conformal class. Upper bounds for Laplace eigenvalues in terms of other conformal invariants can be also found in~\cite{AH, Ko17}. Using the Weyl law 
$$
\lambda_k(g)\mathit{Vol}_g(M)^{2/m}\sim\frac{4\pi^2}{\omega_m^{2/m}}k^{2/m}\qquad\text{as~ }k\to +\infty,
$$
where $\omega_m$ is the volume of a unit ball in the $m$-dimensional Euclidean space, we may pass to the limit as $k\to +\infty$ in the inequalities in Theorem~\ref{mt} to obtain that $\mathit{Vol}_g(M)\geqslant C_4(m)\rad(g)^{m}$. This inequality is well-known: it is a consequence of standard volume comparison theorems, and is reminiscent to the so-called "volume -- injectivity radius" inequality due to Berger~\cite{Be80}, see the discussion in Section~\ref{prems}. Thus, the collection of inequalities~\eqref{Be:3} can be viewed as a quantized version of the classical geometric inequality.

The proof of Theorem~\ref{mt} builds on the results from~\cite{GNY,CM08} and~\cite{AH}. The key ingredient is a construction of disjoint sets whose measure is carefully controlled by our geometric hypotheses. Though similar ideas, originating in the work by Buser~\cite{Bu79} and Korevaar~\cite{Kor}, have been used in  a few papers recently, see for example~\cite{AH, Ko14, HaKo}, and~\cite{Ko17}, our hypotheses are rather different from the previous work. In particular, we do not use a lower Ricci curvature bound for a background or auxiliary metric, which is so essential in most of the past papers. Our argument is based on the revision of recently developed techniques that allows to obtain a rather neat control of constants in the estimates for the measure of disjoint sets.

\subsection{The Berger inequality for minimal submanifolds}
Now we consider closed Riemannian manifolds $(\Sigma^n,g_\Sigma)$ that can be isometrically immersed into  $(M,g)$ as minimal submanifolds. In the sequel we might endow such a manifold $\Sigma^n$ with another metric $h$, and denote by 
$$
0=\lambda_0(\Sigma^n,h)<\lambda_1(\Sigma^n,h)\leqslant\lambda_2(\Sigma^n,h)\leqslant\ldots\leqslant\lambda_k(\Sigma^n,h)\leqslant\ldots
$$
its Laplace eigenvalues, repeated with respect to multiplicity. Our next result shows that conformal eigenvalue bounds in Theorem~\ref{mt} extend naturally to minimal submanifolds $\Sigma^n\subset M$.

\begin{theorem}
\label{mtm}
Let $(M,g)$ be a closed Riemannian manifold whose all sectional curvatures are not greater than $\delta\geqslant 0$, and $\Sigma^n\subset M$ a closed immersed minimal submanifold of dimension $n$. Then for any Riemannian metric $h$ on $\Sigma^n$ conformal to $g_\Sigma$ its Laplace eigenvalues satisfy the inequalities
$$
\lambda_k(\Sigma^n,h)\mathit{Vol}_{h}(\Sigma^n)^{2/n}\leqslant C_6(n)\left(\frac{\mathit{Vol}_g(\Sigma^n)}{\rad(g)^n}\right)^{1+2/n}k^{2/n}
$$
for any $k\geqslant 1$, where $\rad(g)$ is the ambient quantity $\min\{\inj(g),\pi/(2\sqrt{\delta})\}$, and $C_6(n)$ is a positive constant that depends on the dimension $n$ only. In particular, the Laplace eigenvalues of the metric $g_\Sigma$ satisfy the inequalities
\begin{equation}
\label{Be:4}
\lambda_k(\Sigma^n,g_\Sigma)\leqslant C_6(n)\frac{\mathit{Vol}_g(\Sigma^n)}{\rad(g)^{n+2}}k^{2/n}
\end{equation}
for any $k\geqslant 1$.
\end{theorem}
Similar to the discussion after Theorem~\ref{mt}, we note that even the inequality for the first non-zero Laplace eigenvalue in~\eqref{Be:4} is new. Passing to the limit as $k\to +\infty$ in inequalities~\eqref{Be:4}, we obtain the lower bound for 
\begin{equation}
\label{v:rad}
\mathit{Vol}_g(\Sigma^n)\geqslant C_6(n)\rad(g)^n
\end{equation}
the volume of an immersed minimal submanifold $\Sigma^n$. This geometric inequality can be independently obtained from comparison monotonicity theorems for minimal submanifolds, see the discussion in Section~\ref{prems}. When the sectional curvatures of $M$ are non-positive, inequality~\eqref{v:rad} can be already derived from the work of Anderson~\cite{An82}. When the upper bound $\delta$ for sectional curvatures of $M$ is positive, to our knowledge, it is unknown whether the quantity used by Anderson is monotonic, see~\cite{GS87} for a related discussion. For this case we prove monotonicity of a different quantity, which might be of independent interest. These monotonicity theorems yield two-sided volume bounds for the volumes of extrinsic balls, and play a crucial role in the proof of Theorem~\ref{mtm}.

Theorem~\ref{mtm} can be extended to the case when $M$ is complete, but not necessarily compact. If the injectivity radius $\inj(g)$ of $M$ is positive, then the statement of Theorem~\ref{mtm} continues to hold for closed minimal submanifolds $\Sigma^n\subset M$. If $\inj(g)=0$, then the injectivity radius in the formula for $\rad(g)$ should be replaced by the quantity $\inf\{\inj_p(g):p\in\Sigma^n\}$. If $\Sigma^n$ is  not closed, then one can consider boundary value problems for domains $\Omega\subset\Sigma^n$. In this case the statement of Theorem~\ref{mtm} is amenable to extensions to the Neumann eigenvalue problem. Below we give a sample version of such a result. For the sake of simplicity we assume that the ambient manifold $M$ is a Cartan-Hadamard space, that is a complete simply-connected space with non-positive sectional curvatures. First, we introduce more notation.

Let $\Sigma^n$ be a complete minimal submanifold in a Cartan-Hadamard space $M$. By the monotonicity theorem of Anderson~\cite{An82} the ratio $\mathit{Vol}(B(p,r)\cap\Sigma^n)/(\omega_nr^n)$ is a non-decreasing function of $r>0$, where $B(p,r)$ is a ball of radius $r$ in $M$, and $\omega_n$ is the volume of  a unit ball in the Euclidean space $\mathbb R^n$. By $\theta(\Sigma^n)$ we denote the (possibly infinite) quantity
$$
\theta(\Sigma^n)=\lim_{r\to +\infty}\frac{\mathit{Vol}_g(B(p,r)\cap\Sigma^n)}{\omega_nr^n};
$$
it does not depend on a reference point $p\in M$, and is called the {\em density at infinity} of $\Sigma^n$. We have the following version of Theorem~\ref{mtm}.
\begin{theorem}
\label{mtm:extra}
Let $(M,g)$ be a Cartan-Hadamard manifold, and $\Sigma^n\subset M$ a complete properly immersed minimal submanifold. Then for any precompact domain $\Omega\subset\Sigma^n$ and any Riemannian metric $h$ on $\Omega$ conformal to $g_\Sigma$ its Neumann eigenvalues satisfy the inequalities
$$
\lambda_k(\Omega,h)\mathit{Vol}_{h}(\Omega)^{2/n}\leqslant C_7(n)\theta(\Sigma^n)^{1+2/n}k^{2/n}
$$
for any $k\geqslant 1$, where $C_7(n)$ is a positive constant that depends on the dimension $n$ only.
\end{theorem}
We end this discussion on the Neumann problem with the following two remarks. First, when $M$ is a Euclidean space $\mathbb R^m$, there is an abundance of examples when $\theta(\Sigma^n)$ is finite -- this is always the case when $\Sigma^n$ has finite total curvature. More precisely, by the classical results of Osserman~\cite{Os63,Os64}, Chern and Osserman~\cite{CO67}, and Anderson~\cite{An84}, such manifolds have finite topological type, that is, they are diffeomorphic to smooth compact manifolds with finitely many points removed. These points correspond to the ends of a minimal submanifold $\Sigma^n$, and the density at infinity $\theta(\Sigma^n)$ coincides with their number counted with multiplicity. When $n\geqslant 3$, by~\cite{An84} each end of $\Sigma^n$ is embedded and its multiplicity equals one. In other words, when $n\geqslant 3$, the density at infinity of such a minimal submanifolds is precisely the number of ends. Thus, Theorem~\ref{mtm:extra} yields {\em topological eigenvalue bounds} for domains in minimal submanifolds  $\Sigma^n\subset\mathbb R^m$ of finite total curvature.

Second, to our knowledge, no upper bounds for Neumann eigenvalues of domains in minimal submanifolds $\Sigma^n\subset\mathbb R^m$ is known until now, unless $\Sigma^n$ is an affine subspace. The situation is in contrast with the Dirichlet problem, where (in this case more natural lower) bounds for the Dirichlet eigenvalues have been known since 1984, see~\cite{CLY84,Ma86}. Thus, Theorem~\ref{mtm:extra} gives an answer to the question that appears to have been open for some time.

\subsection{Ambient bounds for Laplace eigenvalues of minimal submanifolds}
There is another version of Theorem~\ref{mtm} that leads to bounds for Laplace eigenvalues of minimals submanifolds in terms of geometry of the ambient space.

\begin{theorem}
\label{tma1}
Let $(M,g)$ be a closed Riemannian manifold whose all sectional curvatures are not greater than $\delta\geqslant 0$, and $\Sigma^n\subset M$ a closed immersed minimal submanifold of dimension $n$. Then for any Riemannian metric $h$ on $\Sigma^n$ conformal to $g_\Sigma$ its Laplace eigenvalues satisfy the inequalities
$$
\lambda_k(\Sigma^n,h)\mathit{Vol}_{h}(\Sigma^n)^{2/n}\leqslant C_4(m)\left(\frac{\mathit{Vol}_g(M)}{\rad(g)^{m+2}}\right){\mathit{Vol}_g(\Sigma^n)}^{2/n}k^{2/n}
$$
for any $k\geqslant 1$, where $\rad(g)$ is the ambient quantity $\min\{\inj(g),\pi/(2\sqrt{\delta})\}$, and $C_4(m)$ is the constant from Theorem~\ref{mt}. In particular, the Laplace eigenvalues of the metric $g_\Sigma$ satisfy the inequalities
\begin{equation}
\label{Be:5}
\lambda_k(\Sigma^n,g_\Sigma)\leqslant C_4(m)\frac{\mathit{Vol}_g(M)}{\rad(g)^{m+2}}k^{2/n}
\end{equation}
for any $k\geqslant 1$.
\end{theorem}

We proceed with one more related result. It also gives eigenvalue bounds in terms of geometry of the ambient space, but has an extra, more traditional, hypothesis -- we additionally assume that the Ricci curvature of the ambient space is bounded below.

\begin{theorem}
\label{tma2}
Let $(M,g)$ be a closed Riemannian manifold whose all sectional curvatures are not greater than $\delta\geqslant 0$, and Ricci curvature is bounded below, $\mathit{Ricci}\geqslant -(m-1)\kappa$, where $\kappa\geqslant 0$. Let $\Sigma^n\subset M$ be a closed immersed minimal submanifold of dimension $n$. Then for any Riemannian metric $h$ on $\Sigma^n$ conformal to $g_\Sigma$ its Laplace eigenvalues satisfy the inequalities
$$
\lambda_k(\Sigma^n,h)\mathit{Vol}_{h}(\Sigma^n)^{2/n}\leqslant C_8(m)\max\{\kappa,\rad(g)^{-2}k^{2/n}\}{\mathit{Vol}_g(\Sigma^n)}^{2/n} 
$$
for any $k\geqslant 1$, where $\rad(g)$ is the ambient quantity $\min\{\inj(g),\pi/(2\sqrt{\delta})\}$, and $C_8(m)$ is a positive constant that depends on the dimension $m$ only. In particular, the Laplace eigenvalues of the metric $g_\Sigma$ satisfy the inequalities
\begin{equation}
\label{Be:6}
\lambda_k(\Sigma^n,g_\Sigma)\leqslant C_8(m)\max\{\kappa,\rad(g)^{-2}k^{2/n}\}
\end{equation}
for any $k\geqslant 1$.
\end{theorem}
To our knowledge, Theorems~\ref{tma1} and~\ref{tma2} are first results in the literature that give upper bounds for Laplace eigenvalues in terms of ambient geometry. Previously, spectral properties (mostly related to the first non-zero eigenvalue) of minimal submanifolds have been studied in rank one symmetric  spaces only, see~\cite{LiYau82,EHI09,Ko17} and references therein. Note also that any complex submanifold of a K\"ahler manifold is minimal, and hence, the theorems above yield eigenvalue bounds for all complex submanifolds in terms of geometry of the ambient K\"ahler manifold. It is extremely interesting to know whether such upper bounds for complex submanifolds can be extended to all K\"ahler metrics with cohomologous K\"ahler forms. For projective submanifolds such results are obtained in~\cite{Ko18}.

Concerning lower bounds for minimal submanifolds, we mention the following result due to Cheng and Tysk in~\cite{CT94}:  for any closed minimal submanifold $\Sigma^n\subset M$ its Laplace eigenvalues satisfy the inequalities
$$
C(n,M)k^{2/n}\leqslant\lambda_k(\Sigma^n,g)\mathit{Vol}_g(\Sigma^n)^{2/n}
$$
for any $k\geqslant \bar C(n,M)\mathit{Vol}_g(\Sigma^n)$, where $C(n,M)$ and $\bar C(n,M)$ are positive constants that depend on the dimension $n$ of $\Sigma^n$ and the geometry of $M$ in a rather implicit way. It is important to note that, in contrast with these lower bounds, the scale-invariant quantities $\lambda_k(\Sigma^n,g)\mathit{Vol}_g(\Sigma^n)^{2/n}$ {\em can not be bounded above} in terms of the ambient geometry only. To see this, recall that by~\cite{CM19} for any so-called bumpy metric $g$ on a closed ambient manifold $M$ of dimension $m$, where $3\leqslant m\leqslant 7$, there is a sequence of closed connected embedded minimal hypersurfaces $\{\Sigma_i^{m-1}\}$ whose volumes tend to $+\infty$. As is known~\cite{Wh91,Wh17}, bumpy metrics form a dense subset in the set of all metrics on $M$, and in particular, we may choose a bumpy metric $g$ of positive Ricci curvature.  Then, by the result of Choi and Wang~\cite{CW83}, we conclude that
$$
\lambda_1(\Sigma_i^{m-1},g)\mathit{Vol}_g(\Sigma_i^{m-1})^{2/(m-1)}\geqslant C\mathit{Vol}_g(\Sigma_i^{m-1})^{2/(m-1)}\rightarrow +\infty\qquad\text{when}\quad i\to+\infty,
$$
where $C>0$ is a constant that depends on the lower bound for the Ricci curvature. Thus, no ambient upper bound for $\lambda_k(\Sigma^n,g)\mathit{Vol}_g(\Sigma^n)^{2/n}$ for any $k\geqslant 1$ may exist.

We end with a brief discussion of the following corollary of Theorem~\ref{tma2}, which gives particularly simple estimates for Laplace eigenvalues of minimal submanifolds in certain positively curved spaces.
\begin{corollary}
\label{cm}
Let $(M,g)$ be a compact Riemannian manifold such that one of the following holds:
\begin{itemize}
\item [(i)] either $M$ is even-dimensional and its sectional curvatures satisfy the bounds
$$
0<K_p(\sigma)\leqslant\delta\qquad\text{for any plane }\sigma\in T_pM,
$$
for any point $p\in M$;

\item [(ii)] or $M$ is simply connected and has $\frac{1}{4}$-pinched sectional curvatures,
$$
\frac{1}{4}\delta\leqslant K_p(\sigma)\leqslant\delta\qquad\text{for any plane }\sigma\in T_pM,
$$
for any point $p\in M$, where $\delta>0$.
\end{itemize}
Let $\Sigma^n\subset M$ be a closed immersed minimal submanifold. Then for any Riemannian metric $h$ on $\Sigma^n$ conformal to $g_\Sigma$ its Laplace eigenvalues satisfy the inequalities
$$
\lambda_k(\Sigma^n,h)\mathit{Vol}_h(\Sigma^n)^{2/n}\leqslant C_9(m)\delta\mathit{Vol}_g(\Sigma^n)^{2/n}k^{2/n}
$$
for any $k\geqslant 1$, where $C_9(m)$ is a positive constant that depends on the dimension $m$ of $M$ only. In particular, the Laplace eigenvalues of a metric $g_\Sigma$ satisfy the inequalities
$$
\lambda_k(\Sigma^n,g)\leqslant C_9(m)\delta k^{2/n}
$$
for any $k\geqslant 1$.
\end{corollary}

Corollary~\ref{cm} is a direct consequence of Theorem~\ref{tma2} and Klingenberg's bounds for the injectivity radius, see~\cite{AM97,Pet}. The most significant difference between the two cases in it is the pinching condition that appears in the latter. Note that it imposes strong topological restrictions on $M$: the universal cover of $M$ has to be diffeomorphic to a compact symmetric space of rank one, see~\cite{AM97,BS09}. As examples with geodesics and minimal tori in the Berger spheres show, when $M$ is odd-dimensional, this condition is essential for an upper bound for the Laplace eigenvalues. More generally, the statements above suggest that the relationship between the injectivity radius of $M$ and Laplace eigenvalues of minimal submanifolds might be interesting on its own. In dimension one, it traces to the classical relationship between the injectivity radius and the lengths of closed geodesics, see~\cite{Pet}.

\subsection{Organisation of the paper}
The paper is organised in the following way. In Section~\ref{prems} we discuss volume comparison theorems, and closely related volume monotonicity theorems for minimal submanifolds in Riemannian manifolds whose sectional curvatures are bounded above. In Section~\ref{revisit} we revisit the recent constructions, due to~\cite{GNY,CM08,AH}, of disjoint sets with controlled amount of measure in pseudo-metric spaces. The improvements obtained there are necessary for our main results. The proofs of Theorems~\ref{mt}--\ref{tma2} are collected in Section~\ref{proofs}. The main arguments in all proofs follow the same strategy, but use different ingredients from previous Sections~\ref{prems} and~\ref{revisit}. There is also a certain logical dependence between the proofs of different statements --  in one of them we may refer to the notation or argument used in another. The paper has a short appendix, where we prove inequalities~\eqref{Cro:2}, extending to higher Laplace eigenvalues the inequality for the first eigenvalue found by Croke~\cite{Cro80} in 1980.

%\smallskip
%\noindent
%{\em Acknowledgements.} 

\section{Preliminaries}
\label{prems}
\subsection{Volume comparison and its consequences}
Let $(M,g)$ be a complete $m$-dimensional Riemannian manifold whose sectional curvatures are not greater than $\delta$, where $\delta\in\mathbb R$. We start with recalling the background material on volume comparison theorems for such manifolds. First, we introduce the necessary notation. Below by $\sn_\delta$ we denote the real-valued function given by the formula
\begin{equation}
\label{sn}
\sn_\delta(t)=\left\{
\begin{array}{lc}
({1}/{\sqrt{\delta}}){\sin(\sqrt{\delta}t)}, & \text{ ~if~ }\delta>0,\\
t, & \text{ ~if~ }\delta=0,\\
({1}/{\sqrt{\lvert\delta\rvert}}){\sinh(\sqrt{\lvert\delta\rvert}t)}, & \text{ ~if~ }\delta<0.
\end{array}
\right.
\end{equation}
Then for any $0<r<\pi/\sqrt{\delta}$ we have the following relations for the volumes of a geodesic sphere  and a geodesic ball of radii $r$ in a simply connected $m$-dimensional space of constant sectional curvature $\delta$:
\begin{equation}
\label{avs}
A_\delta(r)=m\omega_m\sn_\delta^{m-1}(r),\qquad V_\delta(r)=m\omega_m\int\limits_0^r\sn_\delta^{m-1}(t)dt,
\end{equation}
where $\omega_m$ is the volume of a unit ball in the $m$-dimensional Euclidean space. We always assume that $\pi/\sqrt{\delta}=+\infty$ when $\delta$ is non-positive. 

Let $(t,\xi)$ be geodesic spherical coordinates around a point $p\in M$, where $t\in (0,\inj_p)$ and $\xi$ is a unit vector in $T_pM$. Let $A_p(t,\xi)$ be the density of the volume measure in these coordinates, that is
$$
%\label{a2exp}
A_p(t,\xi)=t^{m-1}\det(D_{t\xi}\exp_p),
$$
where $\exp_p:T_pM\to M$ is the exponential map, see~\cite{Cha}. Recall that the G\"unther-Bishop comparison theorem~\cite[Theorem~III.4.1]{Cha} says that the function
$$
t\longmapsto\frac{A_p(t,\xi)}{\sn_\delta^{m-1}(t)},\qquad\text{where }\quad 0<t<\min\left\{\inj_p(g),\frac{\pi}{\sqrt{\delta}}\right\},
$$
is non-decreasing for any unit vector $\xi\in T_pM$. The following statement is a consequence of this result, which does not seem to appear explicitly in the literature. For reader's convenience we sketch a proof below.
\begin{prop}
\label{gb}
Let $(M,g)$ be a complete $m$-dimensional Riemannian manifold whose sectional curvatures are not greater than $\delta$, where $\delta\in\mathbb R$. Then for any point $p\in M$ the function
$$
r\longmapsto\frac{1}{V_\delta(r)}\mathit{Vol}_g(B(p,r)),\qquad\text{where }\quad 0<r<\min\left\{\inj_p(g),\frac{\pi}{\sqrt{\delta}}\right\},
$$
is non-decreasing. Besides, if it equals one at some value $r$, then the ball $B(p,r)$ is isometric to a ball of radius $r$ in the space form of constant curvature $\delta$.
\end{prop}
\begin{proof}
Integrating the function $A_p(t,\xi)/\sn_\delta^{m-1}(t)$ over unit vectors $\xi\in T_pM$, by the G\"unther-Bishop theorem we conclude that the function
$$
t\longmapsto\frac{\mathit{Area}(S(p,t))}{A_\delta(t)},\qquad\text{where }\quad 0<t<\min\left\{\inj_p(g),\frac{\pi}{\sqrt{\delta}}\right\},
$$
and $S(p,t)$ is a geodesic sphere of radius $t$, is non-decreasing. Now note that if for positive real-valued functions $f(t)$ and $g(t)$ of one variable the ratio $f/g$ is a non-decreasing function, then the ratio $\int_0^rf/\int_0^rg$ is also a non-decreasing function. Taking as $f(t)$ the function $\mathit{Area}(S(p,t))$, and as $g(t)$ the function $A_\delta(t)$, we arrive at the first statement of Proposition~\ref{gb}. The second statement -- the equality case -- follows from the equality case in the standard volume comparison theorem~\cite[Theorem~III.4.2]{Cha}. 
\end{proof}

Recall that a classical result by Berger~\cite{Be80} says that for any closed $m$-dimensional Riemannian manifold the inequality
$$
%\label{berger}
\mathit{Vol}_g(M)\geqslant (m+1)\omega_{m+1}(\inj(g)/\pi)^m,
$$
holds, and the equality occurs if and only if after rescaling $M$ is isometric to the unit round sphere. As a direct consequence of the volume comparison theorems, we also have the comparison version of this result:
\begin{equation}
\label{bec}
\mathit{Vol}_g(M)\geqslant V_\delta(\rad(g)),
\end{equation}
where $\rad(g)$ is $\min\{\inj(g),\pi/(2\sqrt{\delta})\}$, and the function $V_\delta(\cdot)$ is given by the second relation in~\eqref{avs}. This is a sharper inequality, if $\delta\leqslant 0$. One can also characterise the case of equality -- it occurs if and only if $\delta>0$ and after scaling $M$ is isometric to the unit round sphere.

For the sequel we need the following consequence of the volume comparison theorems.
\begin{corollary}
\label{gb:c}
Let $(M,g)$ be a closed $m$-dimensional Riemannian manifold whose sectional curvatures are not greater than $\delta$, where $\delta\geqslant 0$. Then for any point $p\in M$ the volume of a geodesic ball $B(p,r)$ satisfies the inequalities
\begin{equation}
\label{gb:eq2}
{2^{1-m}}\omega_mr^m\leqslant\mathit{Vol}_g(B(p,r))\leqslant 2^{m-1}\frac{\mathit{Vol}_g(M)}{\rad(g)^m}r^m,
\end{equation}
where $\rad(g)$ stands for $\min\{\inj(g),\pi/(2\sqrt{\delta})\}$, and $0<r\leqslant\rad(g)$.
\end{corollary}
\begin{proof}
Indeed, by Proposition~\ref{gb} we obtain
\begin{equation}
\label{gb:in}
1\leqslant\frac{\mathit{Vol}_g(B(p,r))}{V_\delta(r)}\leqslant\frac{\mathit{Vol}_g(B(p,\rad))}{V_\delta(\rad)}
\end{equation}
for any $0<r\leqslant\rad(g)$, where $\rad=\rad(g)$. When $\delta=0$, we have $V_\delta(r)=\omega_mr^m$, and the statement follows directly from~\eqref{gb:in}. Now suppose that $\delta>0$. Then from the inequalities
$$
\frac{1}{2}t\leqslant\sn_\delta(t)\leqslant t\qquad\quad\text{for any }\quad 0\leqslant t\leqslant\frac{\pi}{2\sqrt{\delta}},
$$
we obtain
$$
\frac{1}{2^{m-1}}\omega_mr^m\leqslant V_\delta(r)\leqslant\omega_m r^m\qquad\text{for any}\quad 0\leqslant r\leqslant\frac{\pi}{2\sqrt{\delta}}.
$$
Combining the last relations with the inequalities in~\eqref{gb:in}, we arrive at the statement of the corollary.
\end{proof}

\subsection{Monotonicity theorems for minimal submanifolds}
Let $\Sigma^n$ be an $n$-dimensional immersed minimal submanifold in a Riemannian manifold $(M,g)$; we assume that the sectional curvatures of $M$ are not greater than $\delta$, where $\delta\in\mathbb R$. As above, we use the notation
\begin{equation}
\label{avs:m}
A^{n-1}_\delta(r)=n\omega_n\sn_\delta^{n-1}(r),\qquad V^n_\delta(r)=n\omega_n\int\limits_0^r\sn_\delta^{n-1}(t)dt,
\end{equation}
where $0<r<\pi/\sqrt{\delta}$, for the volumes of a geodesic sphere and a geodesic ball of radii $r$ in an $n$-dimensional space form of curvature $\delta$.

The following {\em volume monotonicity} theorem can be viewed as an extension of Proposition~\ref{gb} to minimal submanifolds. When $\delta\leqslant 0$, it is due to Anderson~\cite{An82}. For $\delta>0$ the statement appears to be new.
\begin{prop}
\label{gbm}
Let $(M,g)$ be a complete Riemannian manifold whose sectional curvatures are not greater than $\delta$, where $\delta\in\mathbb R$, and let $\Sigma^n$ be an $n$-dimensional properly immersed minimal submanifold in $M$. Then for any point $p\in M$ the following holds:
\begin{itemize}
\item [(i)] if $\delta\leqslant 0$, the function
$$
r\longmapsto \frac{1}{V^n_\delta(r)}\mathit{Vol}_g(B(p,r)\cap\Sigma^n),\qquad\text{where}\quad 0<r<\inj_p(g),
$$
is non-decreasing;
\item [(ii)] if $\delta> 0$, the function
$$
r\longmapsto \frac{1}{A^n_\delta(r)}\mathit{Vol}_g(B(p,r)\cap\Sigma^n),\qquad\text{where}\quad 0<r<\min\left\{\inj_p(g),\frac{\pi}{\sqrt{\delta}}\right\},
$$
is non-decreasing.
\end{itemize}
\end{prop}

\begin{remark}
Under the hypotheses of Proposition~\ref{gbm}, consider the case when $\delta>0$. 
To our knowledge, the answer to the following question, also implicitly raised in~\cite{GS87}, is unknown: is the function
$$
r\longmapsto \frac{1}{V^n_\delta(r)}\mathit{Vol}_g(B(p,r)\cap\Sigma^n),\qquad\text{where}\quad 0<r<\min\left\{\inj_p(g),\frac{\pi}{\sqrt{\delta}}\right\},
$$
non-decreasing?
\end{remark}
Proposition~\ref{gbm} immediately implies comparison inequalities for the volumes of extrinsic balls $B(p,r)\cap\Sigma^n$; see also~\cite{CLY84,Ma86}, where these inequalities are obtained from the heat kernel comparison theorems. In particular, we obtain volume bounds for any immersed minimal submanifold $\Sigma^n\subset M$; for example, if $\delta\leqslant 0$, then
\begin{equation}
\label{becm}
\mathit{Vol}_g(\Sigma^n)\geqslant V^n_\delta(\rad(g)).
\end{equation}
By the results in~\cite{CLY84,Ma86} the above inequality continues to hold for the case $\delta>0$ also, while Proposition~\ref{gbm} gives a weaker result in this case. Inequality~\eqref{becm} can be viewed as a version of comparison inequality~\eqref{bec}, inherited by minimal submanifolds. Proposition~\ref{gbm} also implies the following version of Corollary~\ref{gb:c}.
\begin{corollary}
\label{gbm:c}
Let $(M,g)$ be a complete Riemannian manifold whose sectional curvatures are not greater than $\delta$, where $\delta\geqslant 0$, and let $\Sigma^n$ be an $n$-dimensional immersed closed minimal submanifold in $M$. Then for any point $p\in\Sigma^n$ the volume of an extrinsic ball $B(p,r)\cap\Sigma^n$ in $\Sigma^n$ satisfies the inequalities
\begin{equation}
\label{gbm:eq2}
{2^{-n}}n\omega_nr^n\leqslant\mathit{Vol}_g(B(p,r)\cap\Sigma^n)\leqslant 2^{n}\frac{\mathit{Vol}_g(\Sigma^n)}{\rad(g)^n}r^n,
\end{equation}
where $\rad(g)$ stands for $\min\{\inj(g),\pi/(2\sqrt{\delta})\}$, and $0<r\leqslant\rad(g)$.
\end{corollary}
The proof of Corollary~\ref{gbm:c} follows the line of the argument in the proof of Corollary~\ref{gb:c}. The rest of the section is devoted to the proof of Proposition~\ref{gbm}. Our argument borrows some observations from the exposition in~\cite{LMMV}, where the authors describe other monotonic quantities for the case $\delta\leqslant 0$. Let us also mention that for the case $\delta>0$ a monotonic quantity different from the one in Proposition~\ref{gbm} is used in~\cite{GS87}.

We start with a number of auxiliary lemmas. The first statement underlines the difference in the cases $\delta\leqslant 0$ and $\delta>0$. Its proof is elementary, and therefore, is omitted. 
\begin{lemma}
\label{sn:rel}
For any positive integer $n$ the function $\sn_\delta(r)$, defined by~\eqref{sn}, satisfies the following relations:
\begin{itemize}
\item [(i)] if $\delta\leqslant 0$, then
$$
(n-1)\sn'_\delta(r)\int\limits_0^r\sn_\delta^{n-1}(t)dt\leqslant\sn_\delta^n(r)\leqslant n\sn'_\delta(r)\int\limits_0^r\sn_\delta^{n-1}(t)dt
$$
for any $r>0$;
\item [(ii)] if $\delta>0$, then
$$
n\sn'_\delta(r)\int\limits_0^r\sn_\delta^{n-1}(t)dt\leqslant\sn_\delta^n(r)
$$
for any $0<r<\pi/\sqrt{\delta}$.
\end{itemize}
\end{lemma}

For the sequel we need the following consequence of Lemma~\ref{sn:rel}.
\begin{corollary}
\label{extra:mon}
For any positive integer $n$ the function $\alpha_\delta(r)=V_\delta^n(r)/A_\delta^{n-1}(r)$ is non-decreasing for any $\delta\in\mathbb R$. Moreover, it is concave for $\delta\leqslant 0$ and is convex for $\delta>0$, where $0<r<\pi/\sqrt{\delta}$.
\end{corollary}
\begin{proof}
Differentiating $\alpha_\delta(r)$, we obtain
\begin{equation}
\label{ab}
\alpha'_\delta(r)=1-(n-1)\frac{\sn'_\delta}{\sn^n_\delta}(r)\int\limits_0^r\sn^{n-1}_\delta(t)dt,
\end{equation}
and by Lemma~\ref{sn:rel}, conclude that $\alpha'_\delta(r)\geqslant 0$. To prove the second statement of the corollary it is sufficient to consider the cases when $\delta$ equals $-1$, $0$, and $1$. We give an argument for the case $\delta=1$; the others are considered similarly. A direct computation gives
$$
\alpha''_1(r)=-\frac{n-1}{(\sin r)^{n+1}}\left((\sin r)^{n}\cos r-(n-1)(\cos r)^2\int_0^r(\sin t)^{n-1}dt-\int_0^r(\sin t)^{n-1}dt\right).
$$
Denote by $\omega(r)$ the expression in the brackets on the right hand-side; we claim that it is non-positive,  $\omega(r)\leqslant 0$. Computing its derivative, we obtain
$$
\omega'(r)=2\sin r\left(-(\sin r)^n+(n-1)\cos r\int_0^r(\sin t)^{n-1}dt\right)\leqslant 0
$$
for $0<r<\pi$, where in the last inequality we used Lemma~\ref{sn:rel}. Since $\omega(0)=0$, we conclude that $\omega(r)$ is indeed non-positive, and hence, the function $\alpha''_1(r)$ is non-negative on the interval $(0,\pi)$.
\end{proof}

We proceed with the following consequence of the Hessian comparison theorem.
\begin{lemma}
\label{hess}
Under the hypotheses of Proposition~\ref{gbm}, let $r(x)$ be a distance function $\dist(p,x)$ to a point $p\in M$ restricted to a minimal submanifold $\Sigma^n$. Then the relation
$$
\Delta_{\Sigma^n} r(x)\geqslant\frac{\sn'_\delta}{\sn_\delta}(r(x))(n-\abs{\nabla r(x)}^2)
$$
holds for any $x\in\Sigma^n$ such that $0<r(x)<\min\{\inj_p(g),\pi/\sqrt{\delta}\}$.
\end{lemma}
\begin{proof}
Let $\phi$ be a smooth function on $M$, and $\varphi$ be its restriction to $\Sigma^n$. Note that $\grad_x\varphi$ is the tangential (lying in $T_x\Sigma^n$) component of $\grad_x\phi$, and a straightforward calculation shows that
$$
\Hess_x\phi(X,X)=\Hess_x\varphi(X,X)-\langle\grad_x\phi,B_x(X,X)\rangle
$$
for any vector $X\in T_x\Sigma^n$, where $B_x(\cdot,\cdot)$ is the second fundamental form of $\Sigma_n$. As a consequence of this relation, we obtain the following identity for an arbitrary submanifold $\Sigma^n\subset M$:
$$
\Delta_{\Sigma^n}\varphi(x)=\sum_{i=1}^n\Hess_x\phi(X_i,X_i)+\langle\grad_x\phi,H_x\rangle,
$$
where $H_x$ is the mean curvature vector of $\Sigma^n$ at $x$, and $\{X_i\}$ is an orthonormal basis of $T_x\Sigma^n$. Recall that the Hessian comparison theorem, see~\cite{Pet}, says that
$$
\Hess_xr(V,W)\geqslant\frac{\sn'_\delta}{\sn_\delta}(r(x))\left(\langle V,W\rangle-\langle(\partial/\partial r),V\rangle\langle(\partial/\partial r),W\rangle\right)
$$
for any vectors $V$, $W\in T_xM$, where $0<r(x)<\min\{\inj_x(g),\pi/\sqrt{\delta}\}$. Now combining the last two relations together with the assumption that $\Sigma^n$ is minimal, we arrive at the statement of the lemma.
\end{proof}

Finally we need the following well-known application of the co-area formula. We omit its proof, and refer to~\cite{LMMV,Pa99} where related details can be found.
\begin{lemma}
\label{coarea}
Under the hypotheses of Proposition~\ref{gbm}, the function $V(r)=\mathit{Vol}_g(B(p,r)\cap\Sigma^n)$ is differentiable almost everywhere, and $V'(r)\geqslant\mathit{Vol}_g(\partial B(p,r)\cap\Sigma^n)$.
\end{lemma}

\begin{proof}[Proof of Proposition~\ref{gbm}]
Consider the function
$$
f(r)=\int\limits_0^rV^n_\delta(t)/A^{n-1}_\delta(t)dt,\qquad\text{where}\quad 0<r<\pi/\sqrt{\delta}.
$$
Note that it satisfies the relations
\begin{equation}
\label{aux:ode}
f''(r)+(n-1)\frac{\sn'_\delta}{\sn_\delta}(r)f'(r)=1,\qquad f(0)=0,\quad f'(0)=0.
\end{equation}
Define the function $\psi$ on $B(p,r)\cap\Sigma^n$ by the formula $\psi(x)=f\circ r(x)$, where $r(x)=\dist(p,x)$. Computing the Laplacian of $\psi$, we obtain
\begin{multline}
\label{aux:ineq}
\Delta_{\Sigma^n}\psi=f''(r)\abs{\nabla r}^2+f'(r)\Delta_{\Sigma^n}r\geqslant f''(r)\abs{\nabla r}^2+f'(r)\frac{\sn'_\delta}{\sn_\delta}(r)(n-\abs{\nabla r}^2)\\ =1+ (1-\abs{\nabla r}^2)\left(f'(r)\frac{\sn'_\delta}{\sn_\delta}(r)-f''(r)\right),
\end{multline}
where we used Lemma~\ref{hess} in the inequality above, and identity~\eqref{aux:ode} in the last relation. The term in the brackets on the right hand-side can be re-written in the form
\begin{equation}
\label{aux:term}
f'(r)\frac{\sn'_\delta}{\sn_\delta}(r)-f''(r)=n\frac{\sn'_\delta}{\sn^n_\delta}(r)\int_0^r\sn^{n-1}_\delta(t)dt-1.
\end{equation}
Now we consider cases when $\delta\leqslant 0$ and $\delta>0$ separately.

\noindent
{\em Case~(i).} When $\delta\leqslant 0$, by Lemma~\ref{sn:rel} we see that the quantity in~\eqref{aux:term} is non-negative, and hence, by relation~\eqref{aux:ineq}, we conclude that $\Delta_{\Sigma^n}\psi\geqslant 1$. Using the divergence theorem, we obtain
$$
V(r)=\mathit{Vol}_g(B(p,r)\cap\Sigma^n)\leqslant\!\!\!\int\limits_{B_r\cap\Sigma^n}\!\!\!\!\Delta_{\Sigma^n}\psi d\mathit{Vol}_g=\!\!\!\!\!\int\limits_{\partial B_r\cap\Sigma^n}\!\!\!\!\!\langle\grad\psi,\nu\rangle\leqslant f'(r)\mathit{Vol}(\partial B(p,r)\cap\Sigma^n),
$$
where $\nu$ is a unit normal vector, and we used the relation $\abs{\nabla r}\leqslant 1$ in the last inequality. Note that the use of the divergence theorem above is justified by the hypothesis that $\Sigma^n$ is immersed properly in $M$. Now by Lemma~\ref{coarea}, we get
$$
V(r)\leqslant f'(r)V'(r)=\frac{V^n_\delta(r)}{~A^{n-1}_\delta(r)}V'(r).
$$
The latter inequality is equivalent to the hypothesis that the ratio $V(r)/V^n_\delta(r)$ is a non-decreasing function of $r$, where $0<r<\inj_p(g)$.

\noindent
{\em Case~(ii).} When $\delta>0$, by Lemma~\ref{sn:rel} the quantity in~\eqref{aux:term} is non-positive. Introducing the new notation
$$
\epsilon_\delta(r)=1-n\frac{\sn'_\delta}{\sn^n_\delta}(r)\int_0^r\sn^{n-1}_\delta(t)dt\geqslant 0,
$$
we can re-write relation~\eqref{aux:ineq} in the form
$$
1\leqslant\Delta_{\Sigma^n}\psi+(1-\abs{\nabla r}^2)\epsilon_\delta(r)\leqslant\Delta_{\Sigma^n}\psi+\epsilon_\delta(r).
$$
Further, using Corollary~\ref{extra:mon}, one can conclude that $\epsilon_\delta(r)$ is a non-decreasing function as $r$ ranges in the interval $(0,\pi/\sqrt{\delta})$. The latter can be seen as the consequence of the relation
$$
\epsilon'_\delta(r)=\frac{n}{n-1}\alpha''_\delta(r),
$$
see identity~\eqref{ab}, where $\alpha_\delta$ is a function from Corollary~\ref{extra:mon}. This observation together with the argument in Case~(i) above yields the inequality
$$
V(r)\leqslant \frac{V^n_\delta(r)}{~A^{n-1}_\delta(r)}V'(r)+\epsilon_\delta(r)V(r),
$$
where $V(r)$ is the volume $\mathit{Vol}(B(p,r)\cap\Sigma^n)$. By the definition of $\epsilon_\delta(r)$ we obtain
$$
nV(r)\frac{\sn'_\delta}{\sn^n_\delta}(r)\int_0^r\sn^{n-1}_\delta(t)dt\leqslant V'(r)\frac{1}{\sn^{n-1}_\delta(r)}\int_0^r\sn^{n-1}_\delta(t)dt,
$$
where $0<r<\pi/\sqrt{\delta}$. The latter is equivalent to
$$
\frac{(\sn^n_\delta)'}{\sn^n_\delta}(r)=n\frac{\sn'_\delta}{\sn_\delta}(r)\leqslant\frac{V'(r)}{V(r)},
$$
and we conclude that the ratio $V(r)/\sn^n_\delta(r)$ is non-decreasing.
\end{proof}

\begin{remark}
Note that in the course of the proof of Proposition~\ref{gbm} we established the following isoperimetric inequalities
$$
\frac{~A^{n-1}_\delta(r)}{V^n_\delta(r)}\leqslant\frac{\mathit{Vol}_g(\partial B(p,r)\cap\Sigma^n)}{\mathit{Vol}_g(B(p,r)\cap\Sigma^n)},\qquad\text{when}\quad\delta\leqslant 0,
$$
$$
n\frac{\sn'_\delta}{\sn_\delta}(r)\leqslant\frac{\mathit{Vol}_g(\partial B(p,r)\cap\Sigma^n)}{\mathit{Vol}_g(B(p,r)\cap\Sigma^n)},\qquad\text{~~~when}\quad\delta> 0.
$$
The first inequality has an explicit comparison flavour. Similar results are also obtained, by a different method, in~\cite{Pa99}, but under more restrictive hypotheses -- the author assumes that the intersection $B(p,r)\cap\Sigma^n$ is connected, and a point $p$ lies in $\Sigma^n$.
\end{remark}

\section{Revisiting constructions of disjoint sets in metric measure spaces}
\label{revisit}
\subsection{Covers refinement functions}
In this section we revisit the so-called {\em decomposition theorems}, that is the constructions of disjoint sets in pseudo-metric measure spaces with controlled amount of measure. Such results originate in the work of Buser~\cite{Bu79} and Korevaar~\cite{Kor}, and are essential for obtaining upper bounds for the whole spectrum; see, for example, \cite{CM08, AH, Ko14, CEG13, HaKo}, and~\cite{Ko17, Ko18}. The known constructions rely heavily on covering properties by balls of the underlying pseudo-metric space. For applications it is also important to keep track of the bound for the radii of the balls in covers, and record how refined covers are used. These considerations motivate the definitions below. Throughout this section by $(X,d)$ we denote a separable pseudo-metric space, and $B(p,r)$ stands for an open ball $\{x\in X: d(p,x)<r\}$ in $X$. 
\begin{defin}[Small balls]
A non-decreasing function $N:(1,+\infty)\to \mathbb R^+$ is called the {\em small cover refinement function} for a pseudo-metric space $(X,d)$, if for any $\rho>1$ each ball $B(p,r)$ with $0<r\leqslant 1$ can be covered by at most $N(\rho)$ balls of radius $r/\rho$. 
\end{defin}
\begin{defin}[Arbitrary balls]
A non-decreasing function $N:(1,+\infty)\to \mathbb R^+$ is called the {\em cover refinement function}, if for any $\rho>1$ each pseudo-metric ball $B(p,r)$ with $r>0$ can be covered by at most $N(\rho)$ balls of radius $r/\rho$.
\end{defin}
The distinction between considering covers of arbitrary balls and only small balls is important for our applications, see also~\cite{AH,HaKo}. Note that if for some $\rho_0>1$ each pseudo-metric ball $B(p,r)$ with $0<r\leqslant 1$ can be covered by $N_0$ balls with radius $r/\rho_0$, then each $B(p,r)$ can be covered by $N(\rho)$  balls with radius $r/\rho$ for any $\rho>1$, see~\cite[Lemma~3.4]{GNY}. Moreover, the argument in the proof of ~\cite[Lemma~3.4]{GNY} shows that the number $N(\rho)$ of such balls in the covering can be chosen so that the function $\rho\mapsto N(\rho)$ is non-decreasing. In other words, if such a covering property holds for some $\rho_0>1$, then a small cover refinement function exists. Unlike many previous papers, see for example~\cite{GNY,AH,HaKo,Ko17} and references therein, where the mere fact whether such a covering property holds for some $\rho_0>1$ was used, for our purposes the (small) cover refinement function itself is important. For these reasons we re-state and sharpen some of the key results from~\cite{GNY,AH}. First, we recall the necessary notation.

By an annulus $A$ in a pseudo-metric space $(X,d)$ we mean a subset of the following form
$$
\{x\in X: r\leqslant d(x,a)<R\},
$$
where $a\in X$ and $0\leqslant r<R<+\infty$. These real numbers $r$ and $R$ are often referred to as the {\em inner} and {\em outer} radii respectively, and the point $a$ -- as the centre of an annulus $A$. By $2A$ we denote the annulus
$$
\{x\in X: r/2\leqslant d(x,a)<2R\}.
$$
Recall that a measure $\mu$ on a pseudo-metric space $(X,d)$ is called {\em non-atomic} if for any point $p\in X$ the mass $\mu(B(p,r))\to 0$ as $r\to 0+$. When $(X,d)$ is a metric space, this is equivalent to saying that the measure does not charge a single point in $X$.

The following statement follows by examining the proof of~\cite[Theorem~3.5]{GNY}; it is stated in the form reminiscent to~\cite[Corollary~3.12]{GNY}.
\begin{prop}
\label{prop:1}
Let $(X,d)$ be a separable pseudo-metric space such that all balls $B(p,r)$ are precompact, and $N(\rho)$ a cover refinement function for it. Then for any finite non-atomic measure $\mu$ on $X$ and any positive integer $k$ there exists a collection of $k$ annuli $\{A_i\}$ such  that the annuli $\{2A_i\}$ are pair-wise disjoint and
$$
\mu(A_i)\geqslant\mu(X)/(ck)\qquad\text{for any}\quad 1\leqslant i\leqslant k,
$$
where $c=8N(1600)$.
\end{prop}
Note that the existence of a cover refinement function is one of the hypotheses in Proposition~\ref{prop:1}. We also need a statement with the weaker hypothesis -- the existence of a small cover refinement function. It can be obtained by revisiting~\cite{CM08,AH}. The following proposition is a sharpened version of~\cite[Theorem~2.1]{AH}.
\begin{prop}
\label{prop:2}
Let $(X,d)$ be a separable pseudo-metric space such that all balls $B(p,r)$ are precompact, and $N(\rho)$ a small cover refinement function for it. Then for any finite non-atomic measure $\mu$ on $X$ and any positive integer $k$ there exists a collection of $k$ bounded Borel sets $\{A_i\}$ such that
$$
\mu(A_i)\geqslant\mu(X)/(ck)\qquad\text{for any }\quad 1\leqslant i\leqslant k,
$$
where $c=64N(1600)$, and one of the following possibilities hold:
\begin{itemize}
\item [(i)] either all the $A_i$'a are annuli, and then the annuli $2A_i$ are pair-wise disjoint and their outer radii are not greater than one,
\item [(ii)] or the $r_0$-neighborhoods 
$$
A_i^{r_0}=\left\{x\in X: \dist(x,A_i)\leqslant r_0\right\}
$$
are pair-wise disjoint, where $r_0=1600^{-1}$.
\end{itemize}
\end{prop}
An important new point in Proposition~\ref{prop:2} is the linear dependence of the constant $c$ on the refinement function $N$. The proof of Proposition~\ref{prop:2} follows the idea in~\cite{AH}; it relies on the argument in the proof of~\cite[Theorem~3.5]{GNY} and an improved version of a statement from~\cite{CM08}. We discuss it in more detail at the end of the section. 

Now we consider the main example that is used in the sequel -- pseudo-metric measure spaces with homogeneous bounds on the measure of balls. We describe it in the form of the following lemma; its proof is rather standard, but we include it for the sake of completeness.
\begin{lemma}
\label{l:hom}
Let $(X,d)$ be a pseudo-metric space equipped with a measure $\nu$ such that
$$
C_1r^\alpha\leqslant \nu(B(p,r))\leqslant C_2r^\alpha\qquad\text{for any}\quad p\in X\text{ ~and~ } 0<r\leqslant 3,
$$
where $C_1$, $C_2$, and $\alpha$ are positive constants. Then the function $N(\rho)=(6\rho)^\alpha C_2/C_1$ is a small cover refinement function for $X$. If the above inequalities hold for any $r>0$ and any $p\in X$, then the function $N(\rho)$ is a cover refinement function for $X$.
\end{lemma}
\begin{proof}
We prove the first statement of the lemma; the second statement for arbitrary balls follows by the same argument. For a given value $\rho>1$ and a ball $B(p,r)$ with $0<r\leqslant 1$, let $\{B(p_i,r/(2\rho))\}$ be a maximal collection of disjoint balls of radii $r/(2\rho)$ centred at a point $p_i\in B(p,r)$, where $i=1,\ldots,\ell$. It is straightforward to see that the family of balls $\{B(p_i,r/\rho)\}$, where $i=1,\ldots,\ell$, covers the ball $B(p,r)$. Thus, for a proof of the statement it is sufficient to show that the cardinality $\ell$ of this cover is not greater than $(6\rho)^\alpha C_2/C_1$. 

Let $i_0$ be an index such that the measure $\nu(B(p_{i_0},r/(2\rho)))$ is the least value among all measures $\nu(B(p_{i},r/(2\rho)))$, where $i$ ranges over $1,\ldots,\ell$. Then we obtain
\begin{equation}
\label{eq:3:1}
\ell\nu(B(p_{i_0},r/(2\rho)))\leqslant \sum\limits_{i=1}^\ell\nu(B(p_{i},r/(2\rho)))\leqslant \nu(B(p,2r))
\leqslant\nu(B(p_{i_0},3r)),
\end{equation}
where in the second inequality we used the inclusion $B(p_i,r/(2\rho))\subset B(p,2r)$, and in the third the inclusion $B(p,2r)\subset B(p_{i_0},3r)$. Thus, using the hypotheses on the lemma, we obtain
$$
\ell\leqslant\frac{\nu(B(p_{i_0},3r))}{\nu(B(p_{i_0},r/(2\rho)))}\leqslant\frac{C_2(3r)^\alpha}{C_1(r/2\rho)^\alpha}=(6\rho)^\alpha\frac{C_2}{C_1},
$$
and finish the proof of the first statement.
\end{proof}

\subsection{On the proof of Proposition~\ref{prop:2}}
A new ingredient in the proof of Proposition~\ref{prop:2} is the following improved version of~\cite[Corollary~3.12]{CM08}, see also~\cite[Lemma~2.1]{CEG13}.
\begin{lemma}
\label{l:cm}
Let $(X,d)$ be a separable pseudo-metric space, $r>0$ a real number, and $N$ a positive integer such that any ball of radius $4r$ in $X$ can be covered by $N$ balls of radius $r$. Let $\mu$ be a finite Borel measure, and $k$ a positive integer such that
$$
\mu(B(p,r))\leqslant\frac{\mu(X)}{4Nk}\qquad\text{for any}\quad p\in X.
$$
Then there exists a collection of $k$ bounded Borel subsets $\{A_i\}$ such that
$$
\mu(A_i)\geqslant\frac{\mu(X)}{2Nk}\qquad\text{for any}\quad 1\leqslant i\leqslant k,
$$
and the $r$-neighbourhoods $\{A_i^r\}$'s are pair-wise disjoint.
\end{lemma}
The proof of this lemma is based on the following statement.
\begin{claim}
\label{aux:cl}
Let $(X,d)$ be a separable pseudo-metric space, $r>0$ a real number, and $N$ a positive integer such that any ball of radius $4r$ in $X$ can be covered by $N$ balls of radius $r$. Let $\mu$ be a finite Borel measure, and $\beta<\mu(X)$ a positive real number such that
\begin{equation}
\label{aux:1}
\mu(B(p,r))\leqslant\frac{\beta}{2}\qquad\text{for any}\quad p\in X.
\end{equation}
Then there exist bounded Borel subsets $A\subset D$ in $X$ such that
$$
\beta\leqslant\mu(A)\leqslant\mu(D)\leqslant 2N\beta,
$$
and $\dist(A,D^c)\geqslant 3r$.
\end{claim}
\begin{proof}
For a positive integer $\ell$ let $\mathcal U_\ell$ be the collection of all subsets in $X$ that can be written as unions of at most $\ell$ balls of radius $r$, that is
$$
\mathcal U_\ell=\left\{\bigcup_{j=1}^\ell B(x_j,r): x_1,\ldots, x_\ell\in X\right\}.
$$
By $\xi_\ell$ we denote the supremum $\sup\{\mu(U): U\in\mathcal U_\ell\}$. Note that $\mathcal U_\ell\subset\mathcal U_{\ell+1}$, and hence, the sequence $\xi_\ell$ is non-decreasing, $\xi_\ell\leqslant\xi_{\ell+1}$. Since $X$ is a separable pseudo-metric space, it is straightforward to see that there exists a sequence of subsets $\{U_\ell\}$ such that $U_\ell\in\mathcal U_\ell$, $U_\ell\subset U_{\ell+1}$ for each $\ell$, and $\cup_\ell U_\ell=X$. Thus, we conclude that the sequence $\xi_\ell$ converges to the value $\mu(X)$. Since by~\eqref{aux:1} we have $\xi_1\leqslant\beta/2$, then there exists an integer $k\geqslant 2$ such that
$$
\xi_{k-1}\leqslant\beta<\xi_k.
$$
The second inequality implies that there exists a set $A\in\mathcal U_k$ such that $\mu(U)>\beta$. The set $A$ has the form $\cup B(p_j,r)$ for some points $p_j\in X$, and then we define the set $D\subset X$ as the union
$$
D=\bigcup_{j=1}^k B(p_j,4r).
$$
It is straightforward to see that $\dist(A,D^c)$ is at least $3r$. Thus, for a proof of the claim it remains to show that $\mu(D)\leqslant 2N\beta$.

To prove the last inequality note that each ball $B(p_j,4r)$ can be covered by $N$ balls of radius $r$. Thus, the set $D$ can be covered by $kN$ balls of radius $r$, that is, $D\subset W$, where $W\in\mathcal U_{kN}$. Since $kN\leqslant 2(k-1)N$, we see that $W$ can be represented as the union
$$
W=\bigcup_{j=1}^{2N}W_j,\qquad\text{where}\quad W_j\in\mathcal U_{k-1},
$$
and we obtain
$$
\mu(D)\leqslant\mu(W)\leqslant\sum_{j=1}^{2N}\mu(W_j)\leqslant 2N\xi_{k-1}\leqslant 2N\beta.
$$
Thus, the claim is proved.
\end{proof}

\begin{proof}[Proof of Lemma~\ref{l:cm}]
Equipped with Claim~\ref{aux:cl} we can now prove the lemma, following the line of the argument in~\cite[Section~4]{CEG13}. More precisely, taking $\beta=\mu(X)/(2Nk)$, one can construct inductively $k$ pairs $(A_j,D_j)$, where $1\leqslant j\leqslant k$, such that 
$$
A_j\subset D_j,\qquad \dist(A_j,(\cup_{i\leqslant j}D_i)^c)\geqslant 3r,
$$ 
the inequalities
$$
\beta\leqslant\mu(A_j)\leqslant\mu(D_j)\leqslant 2N\beta=\frac{\mu(X)}{k}
$$
hold, and additionally, $A_j\subset (\cup_{i<j}D)^c$. The above claim is used in the induction step. Then the family $\{A_j\}$ satisfies the conclusion of Lemma~\ref{l:cm}. Indeed, we have
$$
\mu(A_j)\geqslant\beta=\frac{\mu(X)}{2Nk},
$$
and since
$$
\dist(A_l,A_j)\geqslant\dist(A_l,(\cup_{i\leqslant l}D_i)^c)\geqslant 3r
$$
for $l<j$, we see that the $r$-neighbourhoods $\{A_j^r\}$ are pair-wise disjoint.

To make the exposition more self-contained, we describe briefly the induction argument for the existence of such pairs $(A_j,D_j)$. Taking $\beta=\mu(X)/(2Nk)$, by the hypotheses of the lemma we see that Claim~\ref{aux:cl} applies, and there are bounded Borel sets $A_1\subset D_1$ such that
$$
\beta\leqslant\mu(A_1)\leqslant\mu(D_1)\leqslant 2N\beta=\frac{\mu(X)}{k},
$$
and $\dist(A_1,D_1^c)\geqslant 3r$. Now suppose that for $1\leqslant j<k$ the desired pairs $\{(A_i,D_i)\}$, where $i=1,\ldots,j$, are constructed. Denote by $\mu_{j+1}$ the measure on $X$, obtained by restricting $\mu$ to the complement $(\cup_{i\leqslant j}D_i)^c$. Note that for any ball $B(p,r)$ the inequalities
$$
\mu_{j+1}(B(p,r))\leqslant\mu(B(p,r))\leqslant\frac{\mu(X)}{4Nk}=\frac{\beta}{2}
$$
hold. By the induction hypotheses we also have
$$
\mu_{j+1}(X)\geqslant \mu(X)-\sum_{i=1}^j\mu(D_j)\geqslant\mu(X)\left(1-\frac{j}{k}\right)\geqslant\frac{\mu(X)}{k},
$$
and hence, see that
$$
\beta=\frac{\mu(X)}{2Nk}\leqslant\frac{\mu_{j+1}(X)}{2N}<\mu_{j+1}(X). 
$$
Thus, Claim~\ref{aux:cl} applies to the measure $\mu_{j+1}$ on $X$, and there are sets $A\subset D$ in $X$ such that
$$
\beta\leqslant\mu_{j+1}(A)\leqslant\mu_{j+1}(D)\leqslant 2N\beta=\frac{\mu(X)}{k},
$$
and $\dist(A,D^c)\geqslant 3r$. The pair $(A_{j+1},D_{j+1})$ is defined by setting
$$
A_{j+1}=A\cap (\cup_{i\leqslant j}D_i)^c\quad\text{ and }D_{j+1}=D\cap (\cup_{i\leqslant j}D_i)^c.
$$
It is straightforward to check that these sets satisfy the required hypotheses.
\end{proof}

Now the proof of Proposition~\ref{prop:2} follows the scheme in~\cite[Section~2]{AH} with necessary adjustments for the constants involved. It relies on the argument in the proof of~\cite[Theorem~3.5]{GNY} and uses Lemma~\ref{l:cm} above in place of~\cite[Lemma~2.3]{AH}.

\section{Proofs}
\label{proofs}

\subsection{Proof of Theorem~\ref{mt}}
Let $(M,g)$ be an $m$-dimensional Riemannian manifold that satisfies the hypotheses of Theorem~\ref{mt}, and $\dist_g(\cdot,\cdot)$ a distance function on it. Scaling the metric $g$, we may assume that $\rad(g)$ equals three. Then the combination of Lemma~\ref{l:hom} and Corollary~\ref{gb:c} implies that the function
\begin{equation}
\label{mt:0}
N(\rho)=C_{11}(m)\frac{\mathit{Vol}_g(M)}{\rad(g)^m}\rho^m,
\end{equation}
where $C_{11}(m)=24^m/\omega_m$, is a small cover refinement function for the metric space $(M,\dist_g)$. For a given metric $\tilde g$ conformal to $g$, denote by $\mu$ its volume measure $\mathit{Vol}_{\tilde g}$ on $M$. Then by Proposition~\ref{prop:2} for any given positive integer $k$ there exists a collection of $2(k+1)$ bounded Borel sets $\{A_i\}$ such that
\begin{equation}
\label{mt:1/2}
\mu(A_i)\geqslant\mu(M)/(2c(k+1))\geqslant\mu(M)/(4ck)
%\qquad\text{for any }\quad 1\leqslant i\leqslant 2(k+1)
\end{equation}
for all $i=1,\ldots, 2(k+1)$, where $c=64N(1600)$, and one of the following possibilities hold:
\begin{itemize}
\item[(i)] either all the $A_i$'s are annuli, and the annuli $2A_i$'s are pair-wise disjoint and their outer radii are not greater than one,
\item[(ii)] or the $r_0$-neighbourhoods of the $A_i$'s, where $r_0=1600^{-1}$, are pair-wise disjoint.
\end{itemize}
Note that, using formula~\eqref{mt:0}, the estimate for $\mu(A_i)$ in relation~\eqref{mt:1/2} can be re-written in the form
\begin{equation}
\label{mt:5}
\mathit{Vol}_{\tilde g}(A_i)\geqslant\frac{\mathit{Vol}_{\tilde g}(M)}{k}\left(256 (1600)^m C_{11}(m)\mathit{Vol}_g(M)/\rad(g)^m\right)^{-1}
\end{equation}
for all $i=1,\ldots, 2(k+1)$. Now we consider two cases corresponding to the two possibilities~(i) and~(ii) above.

\smallskip
\noindent
{\em Case~(i).} Since the annuli $2A_i$'s are pair-wise disjoint, we have
$$
\sum\limits_{i=1}^{2(k+1)}\mu(2A_i)\leqslant\mu(M),
$$
and hence, there exists at least $(k+1)$ sets $A_i$ such that
\begin{equation}
\label{mt:1}
\mu(2A_i)\leqslant\mu(M)/(k+1)\leqslant\mu(M)/k.
\end{equation}
After reordering, we may assume that the above relation holds for $i=1,\ldots,k+1$. For such an $i$ we denote by $u_i$ the test-function constructed in the following way: it vanishes on the complement of the extrerior annulus $2A_i$, equals one on the interior annulus $A_i=B(a_i,R_i)\backslash B(a_i,r_i)$, and is given by the formula
$$
u_i(x)=\left\{
\begin{array}{lc}
\displaystyle{\frac{2}{r_i}}\dist(x,a_i)-1, & \text{if~ } x\in B(a_i,r_i)\backslash B(a_i,r_i/2),\\
2-\displaystyle{\frac{1}{R_i}}\dist(x,a_i), & \text{if~ } x\in B(a_i,2R_i)\backslash B(a_i,R_i),
\end{array}
\right.
$$
on the complement $2A_i\backslash A_i$. It is straightforward to see that each $u_i$ is a Lipschitz function, and moreover, on the complement $2A_i\backslash A_i$ its gradient satisfies the inequalities
\begin{equation}
\label{mt:2}
\abs{\nabla u_i}\leqslant 2/r_i\qquad\text{on}\quad B(a_i,r_i)\backslash B(a_i,r_i/2),
\end{equation}
\begin{equation}
\label{mt:3}
\abs{\nabla u_i}\leqslant 1/{R_i}\qquad\text{on}\quad B(a_i,2R_i)\backslash B(a_i,R_i).
\end{equation}
Now we estimate the Dirichlet energy of $u_i$ with respect to the metric $\tilde g$. By the H\"older inequality, we obtain
\begin{multline*}
\int_M\abs{\nabla u_i}^2_{\tilde g}d\mathit{Vol}_{\tilde g}\leqslant\mathit{Vol}_{\tilde g}(2A_i)^{1-2/m}
\left(\int_{B(a_i,2R_i)}\abs{\nabla u_i}^m_{\tilde g}d\mathit{Vol}_{\tilde g}\right)^{2/m}\\
=\mathit{Vol}_{\tilde g}(2A_i)^{1-2/m}\left(\int_{B(a_i,2R_i)}\abs{\nabla u_i}^m_{g}d\mathit{Vol}_{g}\right)^{2/m}\\ \leqslant\mathit{Vol}_{\tilde g}(2A_i)^{1-2/m}\left((2/r_i)^m\mathit{Vol}_g(B(a_i,r_i))+(1/R_i)^m\mathit{Vol}_g(B(a_i,2R_i))\right)^{2/m},
\end{multline*}
where in the equality above we used the conformal invariance of $\int\abs{\nabla u}^md\mathit{Vol}$, and in the last relation inequalities~\eqref{mt:2}\! --\! \eqref{mt:3}. Now, since by Proposition~\ref{prop:2} the outer radii satisfy the inequality $2R_i\leqslant 1<\rad(g)$, the volume bounds in Corollary~\ref{gb:c} apply, and we obtain
\begin{multline}
\label{mt:4}
\int_M\abs{\nabla u_i}^2_{\tilde g}d\mathit{Vol}_{\tilde g}\leqslant 16\mathit{Vol}_{\tilde g}(2A_i)^{1-2/m}\left(\mathit{Vol}_g(M)/\rad(g)^m\right)^{2/m}\\ 
\leqslant 16(\mathit{Vol}_{\tilde g}(M)/k)^{1-2/m}\left(\mathit{Vol}_g(M)/\rad(g)^m\right)^{2/m},
\end{multline}
where in the last inequality we used relation~\eqref{mt:1}. Combining inequalities~\eqref{mt:5} and~\eqref{mt:4}, we can now estimate the Rayleigh quotient:
\begin{multline*}
\mathcal R_{\tilde g}(u_i)=\left(\int_M\abs{\nabla u_i}_{\tilde g}^2d\mathit{Vol}_{\tilde g}\right)/\left(\int_Mu_i^2d\mathit{Vol}_{\tilde g}\right)\\ \leqslant C_{12}(m)(\mathit{Vol}_{\tilde g}(M)/k)^{-2/m}\left(\mathit{Vol}_g(M)/\rad(g)^m\right)^{1+2/m}\\
= C_{12}(m)(\mathit{Vol}_{\tilde g}(M))^{-2/m}\left(\mathit{Vol}_g(M)/\rad(g)^m\right)^{1+2/m}k^{2/m},
\end{multline*}
where $i=1,\ldots,k+1$. Since the $u_i$'s form a system of $W^{1,2}$-orthogonal functions, by the variational principle we conclude that
$$
\lambda_k(\tilde g)\mathit{Vol}_{\tilde g}(M)^{2/m}\leqslant C_{12}(m)\left(\mathit{Vol}_g(M)/\rad(g)^m\right)^{1+2/m}k^{2/m}.
$$
Thus, the statement of the theorem is proved in this case.

\smallskip
\noindent
{\em Case~(ii).} Since the $r_0$-neighbourhoods of the $A_i$'s are pair-wise disjoint, as in the first case, we may assume that
\begin{equation}
\label{mt:6}
\mu(A_i^{r_0})\leqslant\mu(M)/k\qquad\text{for any}\quad i=1,\ldots,k+1.
\end{equation}
For such an $i$ we denote by $u_i$ the test-function supported in the $r_0$-neighbourhood $A_i^{r_0}$ that is given by the formula
$$
u_i(x)=\left\{
\begin{array}{ll}
1, & \text{if~ }x\in A_i,\\
1-r_0^{-1}\dist(x,A_i), & \text{if~ }x\in A_i^{r_0}\backslash A_i,
\end{array}
\right.
$$
where $\dist(\cdot, A)$ stands for the distance to a subset $A$.
It is straightforward to see that $u_i$ is a Lipschitz function such that $\abs{\nabla u_i}\leqslant r_0^{-1}$ on $A_i^{r_0}\backslash A_i$. Thus, following the line of argument above, we obtain
\begin{multline*}
\int_M\abs{\nabla u_i}^2_{\tilde g}d\mathit{Vol}_{\tilde g}\leqslant\mathit{Vol}_{\tilde g}(A_i^{r_0})^{1-2/m}
\left(\int_{A_i^{r_0}}\abs{\nabla u_i}^m_{\tilde g}d\mathit{Vol}_{\tilde g}\right)^{2/m}\\
=\mathit{Vol}_{\tilde g}(A_i^{r_0})^{1-2/m}\left(\int_{A_i^{r_0}}\abs{\nabla u_i}^m_{g}d\mathit{Vol}_{g}\right)^{2/m}\leqslant (\mathit{Vol}_{\tilde g}(M)/k)^{1-2/m}\mathit{Vol}_{g}(M)^{2/m}r_0^{-2},
\end{multline*}
where in the last inequality we used relation~\eqref{mt:6}. Recall that by our normalisation assumption, we have
$$
r_0=\frac{1}{1600}=\frac{1}{4800}\rad(g).
$$
Hence, the bound above for the Dirichlet energy of $u_i$ can be re-written in the form
\begin{equation}
\label{mt:7}
\int_M\abs{\nabla u_i}^2_{\tilde g}d\mathit{Vol}_{\tilde g}\leqslant 4800^2(\mathit{Vol}_{\tilde g}(M)/k)^{1-2/m}\left(\mathit{Vol}_{g}(M)/\rad(g)^m\right)^{2/m},
\end{equation}
where $i=1,\ldots,k+1$. Now combining inequalities~\eqref{mt:5} and~\eqref{mt:7}, we arrive at the estimate $$
\mathcal R_{\tilde g}(u_i)\leqslant C_{13}(m)(\mathit{Vol}_{\tilde g}(M))^{-2/m}\left(\mathit{Vol}_g(M)/\rad(g)^m\right)^{1+2/m}k^{2/m},
$$
for all $i=1,\ldots,k+1$. Thus, by the variational principle we conclude that the desired inequalities for the eigenvalues of $\lambda_k(\tilde g)$ hold in this case as well.
\qed

\begin{remark}
Note that choosing the sets $A_i$ in the argument in Case~(ii) more carefully, such that in addition to relation~\eqref{mt:6} the following inequalities hold
$$
\mathit{Vol}_g(A_i^{r_0})\leqslant\mathit{Vol}_g(M)/k\qquad\text{for any}\quad i=1,\ldots,k+1,
$$
one can show that the eigenvalue $\lambda_k(\tilde g)$ is bounded independently of $k$ in this case. However, this observation does not give any improvement to the final result.
\end{remark}

\subsection{Proof of Theorem~\ref{mtm}}
The proof of Theorem~\ref{mtm} follows the strategy used in the proof of Theorem~\ref{mt}. However, the way we use the decomposition theorem, Proposition~\ref{prop:2}, as well as a few ingredients involved, are different. 

Let $(\Sigma^n,g)$ be a manifold isometrically immersed to $M$, via $\imath:\Sigma^n\to M$, as a proper minimal submanifold. Below we denote by $g$ the metric on both manifolds $\Sigma^n$ and $M$. We equip $\Sigma^n$ with a pseudo-metric $\bar d(\cdot,\cdot)$ obtained by restricting the distance function $\dist_g(\cdot,\cdot)$ on $M$ to the image $\imath(\Sigma^n)$. A metric ball $\bar B(\bar p,r)$ in this pseudo-metric can be viewed as the pre-image $\imath^{-1}(B(p,r))$, where $\imath(\bar p)=p$ and $B(p,r)$ is a metric ball in $(M,\dist_g)$. Abusing the notation, it is also denoted by $B(p,r)\cap\Sigma^n$ in Section~\ref{prems}. A measure $\bar\mu$ on $\Sigma^n$ is non-atomic with respect to $\bar d(\cdot,\cdot)$, see Section~\ref{revisit}, if and only if the push-forward measure $\imath_*\bar\mu$ is non-atomic on $M$. Since $\imath:\Sigma^n\to M$ is an immersion, it is straightforward to see that for any metric $h$ on $\Sigma^n$ its volume measure is non-atomic with respect to the pseudo-metric $\bar d(\cdot,\cdot)$. 

As in the proof of Theorem~\ref{mt}, we assume that the metric $g$ on $M$ is scaled such that $\rad(g)$ equals three. Then the combination of Lemma~\ref{l:hom} and Corollary~\ref{gbm:c} implies that the function 
\begin{equation}
\label{mtm:0}
\bar N(\rho)=C_{14}(n)\frac{\mathit{Vol}_g(\Sigma^n)}{\rad(g)^n}\rho^n,
\end{equation}
where $C_{14}(n)=24^n/(n\omega_n)$, is a small cover refinement function for the pseudo-metric space $(\Sigma^n,\bar d)$. Now let $h$ be a metric on $\Sigma^n$ that is conformal to $g$, and $\bar\mu$ its volume measure. By the discussion above, Proposition~\ref{prop:2} applies to the pseudo-metric space $(\Sigma^n,\bar d)$ equipped with $\bar\mu$. Thus, for any positive integer $k$ there exists a collection of $2(k+1)$ bounded Borel sets $\{\bar A_i\}$ in $\Sigma^n$ such that
\begin{equation}
\label{mtm:1/2}
\bar\mu(\bar A_i)\geqslant\bar\mu(\Sigma^n)/(2c(k+1))\geqslant\bar\mu(\Sigma^n)/(4\bar ck),
\end{equation}
for all $i=1,\ldots,2(k+1)$, where $\bar c=64\bar N(1600)$, and one of the following possibilities hold:
\begin{itemize}
\item[(i)] either all the $\bar A_i$'s are annuli for the pseudo-metric $\bar d(\cdot,\cdot)$, and the annuli $2\bar A_i$'s are pair-wise disjoint and their outer radii are not greater than one,
\item[(ii)] or the $r_0$-neighbourhoods of the $\bar A_i$'s, where $r_0=1600^{-1}$, are pair-wise disjoint.
\end{itemize}
Now the cases~(i) and~(ii) can be considered following the line of argument in the proof of Theorem~\ref{mt}. The test-functions are constructed similarly, but using the pseudo-metric $\bar d(\cdot,\cdot)$. A new ingredient in the estimate of their Dirichlet energies is one of the inequalities in Corollary~\ref{gbm:c}. Below we briefly sketch the key points of the argument. In the sequel we use estimate~\eqref{mtm:1/2} for $\bar\mu(\bar A_i)$ in the following form
\begin{equation}
\label{mtm:5}
\mathit{Vol}_{h}(\bar A_i)\geqslant\frac{\mathit{Vol}_{h}(\Sigma^n)}{k}\left(256 (1600)^n C_{14}(n)\mathit{Vol}_g(\Sigma^n)/\rad(g)^n\right)^{-1}.
\end{equation}
It follows by combination of the relation $\bar c=64\bar N(1600)$ with formula~\eqref{mtm:0} for a small cover refinement function.

\smallskip
\noindent
{\em Case~(i).} As in the proof of Theorem~\ref{mt}, we may assume that
\begin{equation}
\label{mtm:1}
\bar\mu(2\bar A_i)\leqslant\bar\mu(\Sigma^n)/(k+1)\leqslant\bar\mu(\Sigma^n)/k
\end{equation}
for $i=1,\ldots,k+1$. For each such $i$ the test-function $\bar u_i$ is set to equal one on the interior annulus $\bar A_i=\bar B(\bar a_i,R_i)\backslash \bar B(\bar a_i,r_i)$ and zero on the complement of the exterior annulus $2\bar A_i$. On the complement $2\bar A_i\backslash \bar A_i$, it is given by the formula
\begin{equation}
\label{barui}
\bar u_i(x)=\left\{
\begin{array}{lc}
\displaystyle{\frac{2}{r_i}}\bar d(x,\bar a_i)-1, & \text{if~ } x\in \bar B(\bar a_i,r_i)\backslash \bar B(\bar a_i,r_i/2),\\
2-\displaystyle{\frac{1}{R_i}}\bar d(x,\bar a_i), & \text{if~ } x\in \bar B(\bar a_i,2R_i)\backslash \bar B(\bar a_i,R_i).
\end{array}
\right.
\end{equation}
It is straightforward to see that $\abs{\nabla\bar d (x,\cdot)}\leqslant 1$ for any point $x\in\Sigma^n$, and hence, the gradient of $\bar u_i$ satisfies the inequalities
$$
%\label{mtm:2}
\abs{\nabla \bar u_i}\leqslant 2/r_i\qquad\text{on}\quad \bar B(\bar a_i,r_i)\backslash \bar B(\bar a_i,r_i/2),
$$
$$
%\label{mtm:3}
\abs{\nabla \bar u_i}\leqslant 1/{R_i}\qquad\text{on}\quad \bar B(\bar a_i,2R_i)\backslash \bar B(\bar a_i,R_i).
$$
Arguing as in the proof of Theorem~\ref{mt}, we can now estimate the Dirichlet energy of $\bar u_i$. In more detail, we obtain
\begin{multline*}
\int_{\Sigma^n}\abs{\nabla \bar u_i}^2_{h}d\mathit{Vol}_{h} \leqslant\mathit{Vol}_{h}(2\bar A_i)^{1-2/n}\left((2/r_i)^n\mathit{Vol}_g(\bar B(\bar a_i,r_i))+(1/R_i)^n\mathit{Vol}_g(\bar B(\bar a_i,2R_i))\right)^{2/n}\\
\leqslant 16\mathit{Vol}_{h}(2\bar A_i)^{1-2/n}\left(\mathit{Vol}_g(\Sigma^n)/\rad(g)^n\right)^{2/n}\\ 
\leqslant 16(\mathit{Vol}_{h}(\Sigma^n)/k)^{1-2/n}\left(\mathit{Vol}_g(\Sigma^n)/\rad(g)^n\right)^{2/n},
\end{multline*}
where we used Corollary~\ref{gbm:c} to estimate volumes of extrinsic balls in the second inequality, and relation~\eqref{mtm:1} in the third. Combining the last inequality with relation~\eqref{mtm:5}, we obtain the following estimate for the Rayleigh quotient of $\bar u_i$:
\begin{multline*}
\mathcal R_{h}(\bar u_i)=\left(\int_{\Sigma^n}\abs{\nabla \bar u_i}_h^2d\mathit{Vol}_{h}\right)/\left(\int_{\Sigma^n}\bar u_i^2d\mathit{Vol}_{h}\right)\\ \leqslant C_{15}(n)(\mathit{Vol}_{h}(\Sigma^n))^{-2/n}\left(\mathit{Vol}_g(\Sigma^n)/\rad(g)^n\right)^{1+2/n}k^{2/n}
\end{multline*}
for any $i=1,\ldots,k+1$. By the variational principle, these estimates immediately yield the desired inequality for the Laplace eigenvalue $\lambda_k(\Sigma^n,h)$.

\smallskip
\noindent
{\em Case~(ii).} As in the proof of Theorem~\ref{mt}, we may assume that
\begin{equation}
\label{mtm:6}
\mu(\bar A_i^{r_0})\leqslant\mu(\Sigma^n)/k\qquad\text{for any}\quad i=1,\ldots,k+1.
\end{equation}
The test-function $\bar u_i$, supported in the $r_0$-neighbourhood $\bar A_i^{r_0}$, is defined by the formula
$$
\bar u_i(x)=\left\{
\begin{array}{ll}
1, & \text{if~ }x\in \bar A_i,\\
1-r_0^{-1}\overline{\dist}(x,\bar A_i), & \text{if~ }x\in \bar A_i^{r_0}\backslash \bar A_i,
\end{array}
\right.
$$
where $\overline{\dist}(\cdot,\bar A)$ is the distance to a subset in the sense of pseudo-metric $\bar d(\cdot,\cdot)$. As above, we see that $\abs{\nabla \bar u_i}\leqslant r_0^{-1}$ on the complement $\bar A_i^{r_0}\backslash \bar A_i$, and estimate its Dirichlet energy in the following way:
\begin{multline*}
\int_{\Sigma^n}\abs{\nabla \bar u_i}^2_{h}d\mathit{Vol}_{h}\leqslant\mathit{Vol}_{h}(\bar A_i^{r_0})^{1-2/n}
\left(\int_{\bar A_i^{r_0}}\abs{\nabla \bar u_i}^n_{g}d\mathit{Vol}_{g}\right)^{2/n}\\
\leqslant (\mathit{Vol}_{h}(\Sigma^n)/k)^{1-2/n}\mathit{Vol}_{g}(\Sigma^n)^{2/n}r_0^{-2}\\
=4800^2(\mathit{Vol}_{h}(\Sigma^n)/k)^{1-2/n}\left(\mathit{Vol}_{g}(\Sigma^n)/\rad(g)^n\right)^{2/n},
\end{multline*}
where we used relation~\eqref{mtm:6} in the second inequality, and the scaling assumption $\rad(g)=3$ together with $r_0=1600^{-1}$ in the last relation. Combining this estimate with relation~\eqref{mtm:5}, we obtain 
$$
\mathcal R_{h}(\bar u_i)\leqslant C_{16}(n)(\mathit{Vol}_{h}(\Sigma^n))^{-2/n}\left(\mathit{Vol}_g(\Sigma^n)/\rad(g)^n\right)^{1+2/n}k^{2/n}
$$
for any $i=1,\ldots,k+1$. Now the desired inequality for the Laplace eigenvalue $\lambda_k(\Sigma^n,h)$ follows from the variational principle.
\qed

\subsection{Proof of Theorem~\ref{mtm:extra}}
As in the proof of Theorem~\ref{mtm}, we consider a pseudo-metric space $(\Sigma^n,\bar d)$, where a pseudo-metric $\bar d(\cdot,\cdot)$ is obtained by restricting the distance function $\dist_g(\cdot,\cdot)$ to the image of an immersed submanifold $\Sigma^n$. For a point $\bar p\in\Sigma^n$ the volume of a pseudo-metric ball $\bar B(\bar p,r)$ satisfies the inequalities
\begin{equation}
\label{bnds:dns}
\omega_nr^n\leqslant\mathit{Vol}_g(\bar B(\bar p,r))\leqslant\omega_n\theta(\Sigma^n)r^n
\end{equation}
for any $r>0$, where $\omega_n$ is the volume of a unit ball in the Euclidean space $\mathbb R^n$, and $\theta(\Sigma^n)$ is the density at infinity. These inequalities are direct consequences of the volume monotonicity for minimal submanifolds, see Proposition~\ref{gbm}. By Lemma~\ref{l:hom} inequalities~\eqref{bnds:dns} imply that the function $\bar N(\rho)=\theta(\Sigma^n)(6\rho)^n$ is a cover refinement function for this pseudo-metric space. 

Let $h$ be a metric conformal to $g$ on a domain $\Omega\subset\Sigma^n$, and $\bar\mu$ its volume measure restricted to $\Omega$. As in the proof of Theorem~\ref{mtm}, we conclude that the measure $\bar\mu$ is non-atomic with respect to $\bar d(\cdot,\cdot)$, and  Proposition~\ref{prop:1} applies. Thus, for any positive integer $k$ there exists a collection of $2(k+1)$ annuli $\{\bar A_i\}$ in $\Sigma^n$ such that the annuli $\{2\bar A_i\}$ are pair-wise disjoint and
$$
\bar\mu(\bar A_i)\geqslant\bar\mu(\Sigma^n)/(2c(k+1))\geqslant\bar\mu(\Sigma^n)/(4\bar ck)
$$
for all $i=1,\ldots, 2(k+1)$, where 
$$
\bar c=8\bar N(1600)=C_{17}(n)\theta(\Sigma^n).
$$
%The last two relations imply that
%$$
%\mathit{Vol}_h(\bar A_i\cap\Omega)\geqslant \frac{\mathit{Vol}_h(\Omega)}{k}(4C_{17}(n)\theta(\Sigma^n))^{-1}
%$$
%for any $1\leqslant i\leqslant 2(k+1)$. 
Let $\bar u_i$ be a test-function constructed as in Case~(i) of the proof of Theorem~\ref{mtm}; it is supported in the annulus $2\bar A_i$. Then, using inequalities~\eqref{bnds:dns} in place of Corollary~\ref{gbm:c}, one can repeat the argument in the proof of Theorem~\ref{mtm} to show that 
$$
\mathcal R_{h}(\bar u_i)=\left(\int_{\Omega}\abs{\nabla \bar u_i}_h^2d\mathit{Vol}_{h}\right)/\left(\int_{\Omega}\bar u_i^2d\mathit{Vol}_{h}\right)\leqslant C_{18}(n)(\mathit{Vol}_{h}(\Omega))^{-2/n}\theta(\Sigma^n)^{1+2/n}k^{2/n}
$$
for some $k+1$ test-functions. Since these test-functions are supported in pair-wise disjoint sets, by the variational principle we obtain the corresponding inequalities for the Neumann eigenvalues $\lambda_k(\Omega,h)$.
\qed

\subsection{Proof of Theorem~\ref{tma1}}
The proof of the theorem uses ingredients from the proofs of both Theorems~\ref{mt} and~\ref{mtm}. The idea is to apply Proposition~\ref{prop:2} to the metric space $(M,\dist_g)$ equipped with the push-forward measure $\mu_*=\imath_*\mathit{Vol}_h$, where $\imath:\Sigma^n\to M$ is an immersion. The test-functions on $\Sigma^n$ are obtained by pulling back the test-functions $u_i$ that are used in the proof of Theorem~\ref{mt}, and their Dirichlet energies are estimated following the line of argument in the proof of Theorem~\ref{mtm}.

In more detail, let $h$ be a metric on $\Sigma^n$ conformal to $g$, and $\mu_*$ the push-forward volume measure $\imath_*\mathit{Vol}_h$. It is straightforward to see that $\mu $ is non-atomic. Scaling the metric $g$ on $M$, we may assume that $\rad(g)$ equals three. Applying Proposition~\ref{prop:2} to the metric space $(M,\dist_g)$, for any positive integer $k$ we obtain a collection of $2(k+1)$ bounded Borel sets $\{A_i\}$  in $M$ such that
\begin{equation}
\label{mta1:1/2}
\mu_*(A_i)\geqslant\mu_*(M)/(4ck)\qquad\text{for all }\quad i=1,\ldots, 2(k+1),
\end{equation}
where $c=64N(1600)$, the function $N(\rho)$ is given by formula~\eqref{mt:0}, and one of the following possibilities hold:
\begin{itemize}
\item[(i)] either all the $A_i$'s are annuli, and the annuli $2A_i$'s are pair-wise disjoint and their outer radii are not greater than one,
\item[(ii)] or the $r_0$-neighbourhoods of the $A_i$'s, where $r_0=1600^{-1}$, are pair-wise disjoint.
\end{itemize}
In the sequel we also use the notation $\bar A_i$ for the Borel set $\imath^{-1}(A_i)$ in $\Sigma^n$. Then, relation~\eqref{mta1:1/2} can be re-written in the form
\begin{equation}
\label{mta1:5}
\mathit{Vol}_h(\bar A_i)\geqslant\frac{\mathit{Vol}_h(\Sigma^n)}{k}(256(1600)^m C_{11}(m)\mathit{Vol}_g(M)/\rad(g)^m)^{-1}
\end{equation}
for all $i=1,\ldots, 2(k+1)$. Now we briefly describe the arguments for the cases~(i) and~(ii), corresponding to the different properties of the sets $A_i$.

\noindent
{\em Case~(i).} As in the proof of Theorem~\ref{mt}, without loss of generality we may assume that
$$
\mu_*(2A_i)\leqslant\mu_*(M)/(k+1)\leqslant\mu_*(M)/k
$$
for all $i=1,\ldots, k+1$. Let $u_i$ be a test-function constructed in Case~(i) in the proof of Theorem~\ref{mt}. By $\bar u_i$ we denote the test-function supported in $2\bar A_i=\imath^{-1}(2A_i)$, given by $\bar u_i=u_i\circ\imath$. Note that the sets $\bar A_i=\imath^{-1}(A_i)$ and $2\bar A_i=\imath^{-1}(2A_i)$ are annuli in the pseudo-metric space $(\Sigma^n,\bar d)$, and using the notation in the proof of Theorem~\ref{mtm}, our test-functions $\bar u_i$ can be also described by formula~\eqref{barui}. In particular, we may repeat the argument in the proof of Theorem~\ref{mtm} to obtain the estimate
$$
\int_{\Sigma^n}\abs{\nabla \bar u_i}^2_{h}d\mathit{Vol}_{h} \leqslant 16(\mathit{Vol}_{h}(\Sigma^n)/k)^{1-2/n}\left(\mathit{Vol}_g(\Sigma^n)/\rad(g)^n\right)^{2/n}
$$
for any $i=1,\ldots,k+1$. Combining the latter with relation~\eqref{mta1:5}, we arrive at the following estimate for the Rayleigh quotient
\begin{multline*}
\mathcal R_{h}(\bar u_i)=\left(\int_{\Sigma^n}\abs{\nabla\bar u_i}^2d\mathit{Vol}_{h}\right)/\left(\int_{\Sigma^n}\bar u_i^2d\mathit{Vol}_{h}\right)\\ \leqslant C_{12}(m)(\mathit{Vol}_{h}(\Sigma^n)/k)^{-2/n}\left(\mathit{Vol}_g(M)/\rad(g)^m\right)\left(\mathit{Vol}_g(\Sigma^n)/\rad(g)^n\right)^{2/n}\\
= C_{12}(m)(\mathit{Vol}_{h}(\Sigma^n))^{-2/n}\left(\mathit{Vol}_g(M)/\rad(g)^{m+2}\right)\mathit{Vol}_g(\Sigma^n)^{2/n}k^{2/n}
\end{multline*}
for any $i=1,\dots,k+1$. Since the $\bar u_i$'s are supported in the pair-wise disjoint sets $2\bar A_i=\imath^{-1}(2A_i)$ in $\Sigma^n$, they form a $W^{1,2}$-orthogonal system, and the inequalities for $\lambda_k(\Sigma^n,h)$ now follow from the variational principle.

\noindent
{\em Case~(ii).} As in the proof of Theorem~\ref{mt}, we may assume that
$$
\mu_*(A_i^{r_0})\leqslant\mu_*(M)/k\qquad\text{for any}\quad i=1,\ldots,k+1.
$$
Let $u_i$ be a test-function constructed in Case~(ii) in the proof of Theorem~\ref{mt}. By $\bar u_i$ we denote the test-function supported in $\bar A_i^{r_0}=\imath^{-1}(A_i^{r_0})$, given by $\bar u_i=u_i\circ\imath$. As above, we see that 
$$
\abs{\nabla\bar u_i}\leqslant\abs{\nabla (u_i\circ\imath)}\leqslant r_0^{-1}\qquad\text{on}\quad \imath^{-1}(A_i^{r_0}\backslash A_i),
$$
and arguing as in the proof of Theorem~\ref{mtm}, we obtain
\begin{multline*}
\int_{\Sigma^n}\abs{\nabla \bar u_i}^2_{h}d\mathit{Vol}_{h}\leqslant (\mathit{Vol}_{h}(\Sigma^n)/k)^{1-2/n}\mathit{Vol}_{g}(\Sigma^n)^{2/n}r_0^{-2}\\ =4800^2(\mathit{Vol}_{h}(\Sigma^n)/k)^{1-2/n}\left(\mathit{Vol}_{g}(\Sigma^n)/\rad(g)^n\right)^{2/n}.
\end{multline*}
Combining the latter with relation~\eqref{mta1:5}, we arrive at the following estimate
\begin{multline*}
\mathcal R_{h}(\bar u_i)\leqslant C_{13}(m)(\mathit{Vol}_{h}(\Sigma^n)/k)^{-2/n}\left(\mathit{Vol}_g(M)/\rad(g)^m\right)\left(\mathit{Vol}_{g}(\Sigma^n)/\rad(g)^n\right)^{2/n}\\
=C_{13}(m)(\mathit{Vol}_{h}(\Sigma^n))^{-2/n}\left(\mathit{Vol}_g(M)/\rad(g)^{m+2}\right)\mathit{Vol}_g(\Sigma^n)^{2/n}k^{2/n}
\end{multline*}
for any $i=1,\ldots,k+1$, and the inequalities for $\lambda_k(\Sigma^n,h)$ now follow from the variational principle.
\qed

\subsection{Proof of Theorem~\ref{tma2}}
As in the proof of Theorem~\ref{tma1}, the strategy is to apply Proposition~\ref{prop:2} to the metric space $(M,\dist_g)$ equipped with the push-forward measure $\mu_*=\imath_*\mathit{Vol}_h$, where $\imath:\Sigma^n\to M$ is an immersion. However, using the lower Ricci curvature bound, we can construct a different, from the one used before, small cover refinement function on $(M,\dist_g)$.

In more detail, a standard application of the Bishop-Gromov relative volume comparison theorem for spaces with a lower Ricci curvature bound, see~\cite{Cha}, yields the inequality
\begin{equation}
\label{bgro}
\frac{\mathit{Vol}_g(B(p,R))}{\mathit{Vol}_g(B(p,r))}\leqslant\left(\frac{R}{r}\right)^me^{(m-1)\sqrt{\kappa}R}
\end{equation}
for any $0<r\leqslant R$, where $B(p,t)$ stands for a metric ball of radius $t>0$ in the space $(M,\dist_g)$. Scaling the metric $g$ on $M$, we may assume that
\begin{equation}
\label{scale}
\min\left\{\frac{1}{\sqrt{\kappa}},\rad(g)\right\}=3.
\end{equation}
Using relation~\eqref{bgro}, we can repeat the argument in the proof of Lemma~\ref{l:hom} to conclude that the function
$$
%\label{crf}
N_0(\rho)=(6\rho)^me^{(m-1)}.
$$
is a small cover refinement function on $(M,\dist_g)$.

Now let $h$ be a metric on $\Sigma^n$ conformal to $g$, and $\mu_*$ be the push-forward measure $\imath_*\mathit{Vol}_h$. As in the proof of Theorem~\ref{tma1}, the measure $\mu_*$ is non-atomic and Proposition~\ref{prop:2} applies to the metric space $(M,\dist_g)$. Thus, for any positive integer $k$ we can find a collection of $3(k+1)$ bounded Borel sets $\{A_i\}$  in $M$ such that
\begin{equation}
\label{mta2:1/2}
\mu_*(A_i)\geqslant\mu_*(M)/(3c(k+1))\geqslant\mu_*(M)/(6ck)
\end{equation}
for all $i=1,\ldots, 3(k+1)$, where $c=64N_0(1600)$, and one of the following possibilities occur:
\begin{itemize}
\item[(i)] either all the $A_i$'s are annuli, and the annuli $2A_i$'s are pair-wise disjoint and their outer radii are not greater than one,
\item[(ii)] or the $r_0$-neighbourhoods of the $A_i$'s, where $r_0=1600^{-1}$, are pair-wise disjoint.
\end{itemize}
Using the notation $\bar A_i$ for the Borel set $\imath^{-1}(A_i)$ in $\Sigma^n$, relation~\eqref{mta2:1/2} can be re-written in the form
\begin{equation}
\label{mta2:5}
\mathit{Vol}_h(\bar A_i)\geqslant\frac{\mathit{Vol}_h(\Sigma^n)}{k}C_{19}(m)
\end{equation}
for all $i=1,\ldots, 3(k+1)$. Now we consider the cases~(i) and~(ii).

\noindent
{\em Case~(i).} As in the proof of Theorem~\ref{mt}, without loss of generality we may assume that
$$
\mu_*(2A_i)\leqslant\mu_*(M)/(k+1)\leqslant\mu_*(M)/k
$$
for all $i=1,\ldots, k+1$. Let $\bar u_i$ be a test-function supported in $\imath^{-1}(2A_i)$ from the proof of Theorem~\ref{tma1}, see Case~(i). As was shown there, the Dirichlet energy of $\bar u_i$ satisfies the inequality
$$
\int_{\Sigma^n}\abs{\nabla \bar u_i}^2_{h}d\mathit{Vol}_{h} \leqslant 16(\mathit{Vol}_{h}(\Sigma^n)/k)^{1-2/n}\left(\mathit{Vol}_g(\Sigma^n)/\rad(g)^n\right)^{2/n}
$$
for any $i=1,\ldots,k+1$; the argument uses the inequality $\rad(g)\geqslant 3$, see relation~\eqref{scale}. 
Combining this estimate with relation~\eqref{mta2:5}, we obtain
\begin{multline*}
\mathcal R_{h}(\bar u_i)=\left(\int_{\Sigma^n}\abs{\nabla\bar u_i}_h^2d\mathit{Vol}_{h}\right)/\left(\int_{\Sigma^n}\bar u_i^2d\mathit{Vol}_{h}\right)\\ \leqslant C_{20}(m)(\mathit{Vol}_{h}(\Sigma^n))^{-2/n}\left(\mathit{Vol}_g(\Sigma^n)/\rad(g)^{n}\right)^{2/n}k^{2/n}
\end{multline*}
for any $i=1,\dots,k+1$. Now by the variational principle we conclude that
$$
\lambda_k(\Sigma^n,h)\mathit{Vol}_h(\Sigma^n)^{2/n}\leqslant C_{20}(m)\rad(g)^{-2}\mathit{Vol}_g(\Sigma^n)^{2/n}k^{2/n}.
$$

\noindent
{\em Case~(ii).} Denote by $\nu$ the push-forward measure $\imath_*\mathit{Vol}_{g}$ on $M$. Since all sets $A_i$ are pair-wise disjoint, we can choose $(k+1)$ sets such that
\begin{equation}
\label{2choices}
\mu_*(A_i^{r_0})\leqslant\mu_*(M)/k\qquad\text{and}\qquad\nu(A_i^{r_0})\leqslant\nu(M)/k.
\end{equation}
Indeed, there exists at least $2(k+1)$ sets such that the first inequalities occur. Among theses sets we can choose further $(k+1)$ sets such that the second inequalities for the measure $\nu$ hold. Without loss of generality, we may assume that both inequalities in~\eqref{2choices} hold for $i=1,\ldots,k+1$. Let $\bar u_i$ be a test-function supported in $\bar A_i^{r_0}=\imath^{-1}(A_i^{r_0})$ from the proof of Theorem~\ref{tma1}, see Case~(ii). Recall that its gradient satisfies the relation $\abs{\nabla\bar u_i}\leqslant r_0^{-1}$ on $\imath^{-1}(A_i^{r_0}\backslash A_i)$. Thus, we obtain
\begin{multline*}
\int_{\Sigma^n}\abs{\nabla \bar u_i}^2_{h}d\mathit{Vol}_{h}\leqslant\mathit{Vol}_{h}(\bar A_i^{r_0})^{1-2/n}
\left(\int_{\bar A_i^{r_0}}\abs{\nabla \bar u_i}^n_{g}d\mathit{Vol}_{g}\right)^{2/n}\\
\leqslant \mathit{Vol}_{h}(\bar A_i^{r_0})^{1-2/n}\mathit{Vol}_{g}(\bar A_i^{r_0})^{2/n}r_0^{-2}\leqslant (\mathit{Vol}_{h}(\Sigma^n)/k)^{1-2/n}(\mathit{Vol}_{g}(\Sigma^n)/k)^{2/n}r_0^{-2}\\
=\frac{4800^2}{k}\mathit{Vol}_{h}(\Sigma^n)^{1-2/n}\mathit{Vol}_{g}(\Sigma^n)^{2/n}\max\{\kappa,\rad(g)^{-2}\},
\end{multline*}
where we used relations~\eqref{2choices} in the third inequality, and the scaling assumption~\eqref{scale}
%$$
%r_0=\frac{1}{1600}=\frac{1}{4800}\min\left\{\frac{1}{\sqrt{\kappa}},\rad(g)\right\}
%$$
in the last equality. Combining this estimate with relation~\eqref{mta2:5}, we obtain
$$
\mathcal R_{h}(\bar u_i) \leqslant C_{21}(m)\mathit{Vol}_{h}(\Sigma^n)^{-2/n}\mathit{Vol}_g(\Sigma^n)^{2/n}\max\{\kappa,\rad(g)^{-2}\}
$$
for any $i=1,\dots,k+1$. Applying the variational principle, we get the inequalities 
$$
\lambda_k(\Sigma^n,h)\mathit{Vol}_h(\Sigma^n)^{2/n}\leqslant C_{21}(m)\mathit{Vol}_g(\Sigma^n)^{2/n}\max\{\kappa,\rad(g)^{-2}\}.
$$
Comparing the latter with the eigenvalue inequalities in Case~(i) above, we conclude that in both cases the Laplace eigenvalues $\lambda_k(\Sigma^n,h)$ satisfy
$$
\lambda_k(\Sigma^n,h)\mathit{Vol}_h(\Sigma^n)^{2/n}\leqslant C_{8}(m)\max\{\kappa,\rad(g)^{-2}k^{2/n}\}\mathit{Vol}_g(\Sigma^n)^{2/n}
$$
for any $k\geqslant 1$, where $C_8(m)$ equals $\max\{C_{20}(m),C_{21}(m)\}$.
\qed

%\section{$L^\infty$-bounds for Laplace eigenfunctions on minimal submanifolds}

\appendix
\section{Appendix: Croke's bounds for higher Laplace eigenvalues}
\label{app:a}

The purpose of this appendix is to give a proof of the following statement.
\begin{prop}
\label{croke:hi}
Let $(M,g)$ be a closed Riemannian manifold of dimension $m$. Then its Laplace eigenvalues $\lambda_k(g)$ satisfy the inequalities
$$
\lambda_k(g)\leqslant C_3(m)\frac{\mathit{Vol}_g(M)^2}{\conv(g)^{2m+2}}k^{2m}
$$
for any $k\geqslant 1$, where $\conv(g)$ is the convexity radius of $(M,g)$, and $C_3(m)$ is the constant that depends on the dimension $m$ only.
\end{prop}
For $k=1$ the inequality in Proposition~\ref{croke:hi} is due to Croke~\cite[Corollary~19]{Cro80}. Its proof is based on the following upper bound for the first Dirichlet eigenvalue of a geodesic ball $B(p,r)$ in $M$:
\begin{equation}
\label{cro:d}
\lambda_0(B(p,r))\leqslant \bar C_3(m)\frac{\mathit{Vol}_g(B(p,r))^2}{r^{2m+2}},
\end{equation}
where $0<r\leqslant\conv(g)$, and $\bar C_3(m)$ is a constant that depends on $m$ only, see~\cite[Theorem~18]{Cro80}. Below we demonstrate how inequality~\eqref{cro:d} can be used to prove Proposition~\ref{croke:hi}.
\begin{proof}[Proof of Proposition~\ref{croke:hi}]
Pick an arbitrary point $p\in M$, and let $q$ be a point from the cut locus of $p$. Thus, we have 
$$
\dist_g(p,q)\geqslant\inj(g)\geqslant\conv(g).
$$
Denote by $L$ the distance $\dist_g(p,q)$, and let $\gamma:[0,L]\to M$ be a shortest unit speed geodesic joining $p$ and $q$. For a given positive integer $k$ consider geodesic balls $B(p_i,r)$, where $r=L/(4k)$, the $p_i$'s are the points $\gamma(iL/(2k))$ on the geodesic $\gamma$, and $i=0,\ldots, 2k$. It is straightforward to see that these balls are pair-wise disjoint, and hence,
$$
\sum_{i=0}^{2k}\mathit{Vol}_g(B(p_i,r))\leqslant\mathit{Vol}_g(M).
$$
Thus, there exists at least $(k+1)$ points $p_i$ such that
$$
\mathit{Vol}_g(B(p_i,r))\leqslant\mathit{Vol}_g(M)/(k+1)\leqslant \mathit{Vol}_g(M)/k.
$$
Combining the last inequality with Croke's inequality~\eqref{cro:d}, we obtain
$$
\lambda_0(B(p_i,r))\leqslant \bar C_3(m)\frac{(\mathit{Vol}_g(M)/k)^2}{r^{2m+2}}\leqslant 4^{2m+2}C_3(m)\frac{\mathit{Vol}_g(M)^2}{\conv(g)^{2m+2}}k^{2m},
$$
where in the last inequality we used the relation $r\geqslant\conv(g)/(4k)$. Now let $\varphi_i$ be a Dirichlet $\lambda_0$-eigenfunction on the ball $B(p_i,r)$ extended to $M$, by setting it to be equal to zero on the complement $M\backslash B(p_i,r)$. The above inequalities show that the Rayleigh quotients on $M$ of at least $(k+1)$ such functions $\varphi_i$ satisfy the inequality
$$
\mathcal R_g(\varphi_i)\leqslant C_3(m)\frac{\mathit{Vol}_g(M)^2}{\conv(g)^{2m+2}}k^{2m}, 
$$
where we set $C_3(m)=4^{2m+2}\bar C_3(m)$. Since the supports of these $\varphi_i$'s are disjoint, by the variational principle we conclude that the desired inequalities for the Laplace eigenvalues $\lambda_k(g)$ hold indeed.
\end{proof}

\end{document}